\nonstopmode
\documentclass[10pt]{amsart}
\usepackage{latexsym}
\usepackage{fancyhdr}
\usepackage{amsmath, amssymb,amsthm}
\usepackage[all]{xy}
\usepackage{pdflscape}
\usepackage{longtable}
\usepackage{rotating}
\usepackage{verbatim}
\usepackage{hyperref}
\usepackage{cleveref}
\usepackage{subfigure}
\usepackage{mathrsfs}
\usepackage{mdwlist}
\usepackage{dsfont}
\usepackage{bbold}
\usepackage{mathtools}
\usepackage{booktabs}
\usepackage{float}

\DeclareFontFamily{U}{mathb}{\hyphenchar\font45}
\DeclareFontShape{U}{mathb}{m}{n}{
      <5> <6> <7> <8> <9> <10> gen * mathb
      <10.95> mathb10 <12> <14.4> <17.28> <20.74> <24.88> mathb12
      }{}
\DeclareSymbolFont{mathb}{U}{mathb}{m}{n}
\DeclareFontSubstitution{U}{mathb}{m}{n}

\let\dot\relax
\DeclareMathAccent{\dot}{0}{mathb}{"39}
\let\ddot\relax
\DeclareMathAccent{\ddot}{0}{mathb}{"3A}
\let\dddot\relax
\DeclareMathAccent{\dddot}{0}{mathb}{"3B}
\let\ddddot\relax
\DeclareMathAccent{\ddddot}{0}{mathb}{"3C}

\theoremstyle{plain}
\newtheorem*{theorem*}{Theorem}
\newtheorem{theorem}{Theorem}[section]
\newtheorem*{lemma*}{Lemma}
\newtheorem{lemma}[theorem]{Lemma}
\newtheorem*{proposition*}{Proposition}
\newtheorem{proposition}[theorem]{Proposition}
\newtheorem*{corollary*}{Corollary}
\newtheorem{corollary}[theorem]{Corollary}
\newtheorem*{claim*}{Claim}
\newtheorem{claim}{Claim}

\newtheorem*{conjecture*}{Conjecture}

\newtheorem*{question*}{Question}
\theoremstyle{definition}
\newtheorem*{definition*}{Definition}
\newtheorem{definition}[theorem]{Definition}
\newtheorem*{example*}{Example}
\newtheorem{example}[theorem]{Example}

\newtheorem*{algorithm*}{Algorithm}
\newtheorem*{remark*}{Remark}
\newtheorem*{remarks*}{Remarks}
\newtheorem{remark}[theorem]{Remark}

\newtheorem*{convention*}{Convention}

\numberwithin{equation}{section}

\def\al{\alpha}
\def\be{\beta}
\def\ga{\gamma}
\def\de{\delta}
\def\ep{\epsilon}

\def\th{\theta}

\def\la{\lambda}

\def\ta{\tau}

\def\ph{\phi}

\def\ch{\chi}
\def\ps{\psi}
\def\om{\omega}

\def\Ga{\Gamma}

\def\Ph{\Phi}
\def\Ps{\Psi}

\def\N{\mathbb{N}}

\def\R{\mathbb{R}}

\def\cD{\mathcal{D}}

\def\cG{\mathcal{G}}

\def\fX{\mathfrak{X}}

\def\p{\partial}

\def\<{\langle}
\def\>{\rangle}
\renewcommand{\o}{\circ}

\def\Id{\on{Id}}
\def\inv{\on{inv}}
\def\comp{\on{comp}}

\let\on=\operatorname

\newcommand{\sr}[1]%
{\ifmmode{}^\dagger\else${}^\dagger$\fi\ifvmode
\vbox to 0pt{\vss
 \hbox to 0pt{\hskip\hsize\hskip1em
 \vbox{\hsize3cm\raggedright\pretolerance10000
 \noindent #1\hfill}\hss}\vss}\else
 \vadjust{\vbox to0pt{\vss%
 \hbox to 0pt{\hskip\hsize\hskip1em%
 \vbox{\hsize3cm\raggedright\pretolerance10000%
 \noindent #1\hfill}\hss}\vss}}\fi%
}

\def\ev{\on{ev}}

\providecommand{\mapsfrom}{\kern.2em%
\setbox0=\hbox{$\leftarrow$\kern-.10em\rule[0.26mm]{0.1mm}{1.3mm}}\box0%
\kern.3em}

\title[On groups of H\"older diffeomorphisms]
{On groups of H\"older diffeomorphisms\\ and their regularity}

\author[D.N.~Nenning]{David Nicolas Nenning}

\address{D.N.~Nenning: Fakult\"at f\"ur Mathematik, Universit\"at Wien, 
Oskar-Morgenstern-Platz~1, A-1090 Wien, Austria}
\email{david.nicolas.nenning@univie.ac.at}

\author[A.~Rainer]{Armin Rainer}

\address{A.~Rainer: 
Fakult\"at f\"ur Mathematik, Universit\"at Wien, 
Oskar-Morgenstern-Platz~1, A-1090 Wien, Austria}
\email{armin.rainer@univie.ac.at}

\begin{document}
	
\begin{abstract}
	We study the set $\cD^{n,\be}(\R^d)$ 
	of orientation preserving diffeomorphisms of $\R^d$ 
	which differ from the identity by a H\"older $C^{n,\be}_0$-mapping, 
	where $n \in \N_{\ge 1}$ and $\be \in (0,1]$. 
	We show that $\cD^{n,\be}(\R^d)$ forms a group, 
	but left translations in $\cD^{n,\be}(\R^d)$ are in general discontinuous. 
	The groups $\cD^{n,\be-}(\R^d) := \bigcap_{\al < \be} \cD^{n,\al}(\R^d)$ (with its natural Fr\'echet topology) 
	and $\cD^{n,\be+}(\R^d) := \bigcup_{\al > \be} \cD^{n,\al}(\R^d)$ (with its natural inductive locally convex topology)
	however are $C^{0,\om}$ Lie groups for any \emph{slowly vanishing} modulus of continuity $\om$.
	In particular, $\cD^{n,\be-}(\R^d)$ is a topological group and a so-called half-Lie group (with smooth right translations). 
	We prove that the H\"older spaces $C^{n,\be}_0$ are ODE closed, in the sense that pointwise 
	time-dependent $C^{n,\be}_0$-vector fields $u$ have unique flows $\Ph$ in $\cD^{n,\be}(\R^d)$. 
	This includes, in particular, all Bochner integrable functions $u \in L^1([0,1],C^{n,\be}_0(\R^d,\R^d))$. 
	For the latter and $n\ge 2$, we show that the flow map $L^1([0,1],C^{n,\be}_0(\R^d,\R^d)) \to C([0,1],\cD^{n,\al}(\R^d))$, 
	$u \mapsto \Ph$, is 
	continuous (even $C^{0,\be-\al}$), for every $\al < \be$.
	As an application we prove that the corresponding Trouv\'e group $\cG_{n,\be}(\R^d)$ 
	from image analysis
	coincides with the connected component of the identity of $\cD^{n,\be}(\R^d)$.
\end{abstract}

\thanks{Supported by FWF-Project P~26735-N25}
\date{\today}

\maketitle

\section{Introduction}

Let $E$ be a Banach space of functions $f : \R^d \to \R^d$ which is continuously embedded in $C^1_0(\R^d,\R^d)$, i.e., 
$C^1$-mappings which vanish together with its first derivative at infinity.
Let $u : [0,1] \times \R^d \to \R^d$ be a \emph{pointwise time-dependent $E$-vector field}, 
i.e., $u(t, \cdot) \in E$ for all $t$,
$u(\cdot, x)$ is measurable for all $x$, and $t \mapsto \|u(t,\cdot)\|_E$ is integrable.
It is well-known 
that the corresponding pointwise flow 
\begin{equation} \label{eq:intro-flow}
	\Ph(t,x) = x + \int_0^t u(s,\Ph(s,x)) \, ds, \quad x \in \R^d,~ t \in [0,1],
\end{equation}
is a $C^1$-diffeomorphism of $\R^d$ at any $t$. The set of all diffeomorphisms $\Ph(1,\cdot)$ at time $1$ which arise 
in this way form a group $\cG_E$, which we call the \emph{Trouv\'e group} of $E$, since this constructions is 
due to Trouv\'e \cite{Trouve95}; details can be found in the book \cite{Younes10}. 

In general, not much is known about the Trouv\'e group. We are especially interested in precise regularity properties 
of its elements. This is intimately related to the question as to whether $E$ is \emph{ODE closed}, i.e., 
$\cG_E \subseteq \Id + E$, and if not, what the \emph{ODE hull} of $E$ is. We define the ODE hull to be the intersection 
of all ODE closed spaces continuously embedded in $C^1_0(\R^d,\R^d)$ and continuously containing $E$, see \Cref{sec:ODEhull}; 
here it is reasonable to allow for 
locally convex spaces (mutatis mutandis) instead of just Banach spaces.

ODE closedness is closely related to stability and continuity or smoothness properties of composition of mappings. 
Indeed, it has been 
widely studied in the context of regular infinite dimensional Lie groups; cf.\ \cite{Glockner16}. 
Our results are not covered by the general theory from \cite{Glockner16}; H\"older (diffeomorphism) groups fail to be Lie groups.

In this paper we explore these questions in the case that $E$ is a global H\"older space 
$C^{n,\be}_0(\R^d,\R^d)$, $n \in \N_{\ge1}$, $\be \in (0,1]$.
In search for an identification of the elements of the corresponding Trouv\'e group 
$\cG_{n,\be}(\R^d)$,
it is natural to look at the set of orientation preserving diffeomorphisms of $\R^d$ which differ from the identity by a 
$C^{n,\be}_0$-mapping, i.e.,
\begin{equation*}
		\mathcal{D}^{n,\beta}(\mathbb{R}^d):= \big\{ \Ph \in \Id+C^{n,\beta}_0(\mathbb{R}^d,\mathbb{R}^d): 
		\det d\Ph(x) > 0 ~\forall x \in \mathbb{R}^d \big\}.
\end{equation*}
We will show that $\mathcal{D}^{n,\beta}(\mathbb{R}^d)$ is a group with respect to composition, but it is not a topological 
group: left translations and inversion are in general not continuous. 
Left translations become continuous if the outer function is slightly more regular: $\ph \mapsto \ps \o (\Id + \ph)$ is 
continuous from $C^{n,\al}_0(\R^d,\R^d) \to C^{n,\al}_0(\R^d,\R^d)$ if $\ps \in C^{n,\be}_0(\R^d,\R^d)$ and $\al < \be$ 
(the same holds if $\ph \in C^{n,1}_0(\R^d,\R^d)$ and $\ps \in C^{n+1,\be}_0(\R^d,\R^d)$). 
Similarly, 
$\Ph \mapsto \Ph^{-1}$ is continuous from $\cD^{n,\be}(\R^d)$ to $\cD^{n,\al}(\R^d)$ if $\al < \be$.
This motivates the definitions
\begin{gather*}
		\mathcal{D}^{n,\beta-}(\mathbb{R}^d):= \big\{ \Ph \in \Id+C^{n,\beta-}_0(\mathbb{R}^d,\mathbb{R}^d): 
		\det d\Ph(x) > 0 ~\forall x \in \mathbb{R}^d \big\},\\
		\mathcal{D}^{n,\beta+}(\mathbb{R}^d):= \big\{ \Ph \in \Id+C^{n,\beta+}_0(\mathbb{R}^d,\mathbb{R}^d): 
		\det d\Ph(x) > 0 ~\forall x \in \mathbb{R}^d \big\},
\end{gather*}
where $C^{n,\beta-}_0 := \bigcap \{C^{n,\al}_0 : 0 < \al <\be \}$ and 
$C^{n,\beta+}_0 := \bigcup \{C^{n,\al}_0 : \be < \al <1 \}$, equipped with the natural projective, resp.\ inductive 
locally convex topology. 
We prove that $\mathcal{D}^{n,\beta \pm}(\mathbb{R}^d)$ 
are $C^{0,\om}$ Lie groups (see \Cref{sec:C0omconvenient}) 
for every \emph{slowly vanishing} modulus of continuity $\om$, i.e., 
\[
	\liminf_{t \downarrow 0} \frac{\om(t)}{t^\ga} > 0 \quad \text{ for all } \ga>0.
\]
This regularity cannot be improved; see \Cref{cor:optimal}.
In particular, $\mathcal{D}^{n,\beta -}(\mathbb{R}^d)$
are topological groups (which remains open for $\mathcal{D}^{n,\beta+}(\mathbb{R}^d)$ since the underlying 
locally convex topology and the $c^\infty$-topology fall apart in this case). 
The right translations are bounded affine linear (in the chart representation) and 
hence smooth. Consequently, $\mathcal{D}^{n,\beta-}(\mathbb{R}^d)$ are also \emph{half-Lie groups} as defined in 
\cite{KrieglMichorRainer14a}.

In the second part of the paper we study flows of time-dependent $C^{n,\be}_0$-vector fields. 
Here we distinguish between:
\begin{enumerate}
 	\item \emph{Pointwise time-dependent $C^{n,\be}_0$-vector fields}, i.e., mappings $u : [0,1] \times \R^d \to \R^d$ such that 
 	$u(t, \cdot) \in C^{n,\be}_0(\R^d, \R^d)$ for all $t \in [0,1]$, $u(\cdot,x)$ is measurable for all $x \in\R^d$, and 
 	$t \mapsto \|u(t, \cdot)\|_{n,\be}$ is integrable. (This corresponds to the notion defined at the beginning of the 
 		introduction.)
 	\item \emph{Strong time-dependent $C^{n,\be}_0$-vector fields}, i.e., Bochner integrable functions 
	$u \in L^1([0,1],C^{n,\be}_0(\R^d, \R^d))$.
\end{enumerate}
The latter notion involves strong measurability which entails that the image $u([0,1])$ is 
essentially separable; a non-trivial condition, since the H\"older spaces $C^{n,\be}_0$ are non-separable. 
If $u$ satisfies (2) then $u^\wedge$ satisfies (1), the converse is false.
(For $f \in Z^{X \times Y}$ we consider $f^\vee \in (Z^Y)^X$ defined by $f^\vee(x)(y) = f(x,y)$, and
with $g \in (Z^Y)^X$ we associate $g^\wedge \in Z^{X \times Y}$ with $g^\wedge (x,y) = g(x)(y)$.)

This deficiency has the effect that Carath\'eodory's solution theory for ODEs on Banach spaces which are Bochner 
integrable in time is not well suited for the H\"older space setting. 
Instead we work with pointwise estimates which has the additional benefit 
that our proofs only require the weaker assumptions in (1).

We show that, for all $n\in \N_{\ge 1}$, $\be \in (0,1]$,  
pointwise time-dependent $C^{n,\be}_0$-vector fields $u$ have unique pointwise flows 
$\Ph \in C([0,1],\cD^{n,\be}(\R^d))$; in particular, $C^{n,\be}_0$ is ODE closed 
(although composition in $\cD^{n,\be}$ is not continuous!). 
As a consequence, for $u \in L^1([0,1],C^{n,\be}_0(\R^d,\R^d))$, 
the identity \eqref{eq:intro-flow} lifts to an identity in $\cD^{n,\al}(\R^d)$, for each $\al < \be$ 
(see \Cref{odesolv}):
\[
	\Ph^\vee(t) = \Id + \int_0^t u(s) \o \Ph^\vee(s) \, ds, \quad  t \in [0,1]. 
\]
Furthermore, we identify the corresponding Trouv\'e group:
\begin{equation} \label{eq:intro1}
	\cG_{n,\be}(\R^d) = \cD^{n,\be}(\R^d)_0, 
\end{equation}
where $\cD^{n,\be}(\R^n)_0$ denotes the connected component of the identity in $\cD^{n,\be}(\R^n)$.
Thus there seems to be no natural topology on $\cG_{n,\be}(\R^d)$ which makes it a topological group.
On the other hand, we also get 
\begin{equation} \label{eq:intro2}
	\cG_{n,\be-}(\R^d) = \cD^{n,\be-}(\R^d)_0, \quad \cG_{n,\be+}(\R^d) = \cD^{n,\be+}(\R^d)_0
\end{equation}
which endows $\cG_{n,\be\pm}(\R^d)$ with a $C^{0,\om}$ 
Lie group structure, for every slowly vanishing modulus of continuity $\om$; 
on $\cG_{n,\be-}(\R^d)$ we also get a topological group structure and a half-Lie group structure.
We wish to point out that our proof of \eqref{eq:intro1} subsequently shows that the equality \eqref{eq:intro1} 
also holds 
if in the definition of the Trouv\'e group $\cG_{n,\be}(\R^d)$ one restricts to 
pointwise time-dependent $C^{n,\be}_0$-vector fields which are piecewise $C^n$ in time; see 
\Cref{rem:Trouve}.

In the third part we investigate the continuity of the flow map $u \mapsto \Ph$. 
We find that as a mapping 
\begin{equation} \label{eq:intro-flowmap1}
	L^1([0,1],C^{n,\be}_0(\R^d,\R^d)) \to C([0,1],\cD^{n,\al}(\R^d))
\end{equation}
the flow map is 
\begin{itemize}
	\item bounded, if $n\in \N_{\ge 1}$ and $0 < \al \le \be \le 1$,
	\item continuous, even $C^{0,\be-\al}$, if $n\in \N_{\ge 2}$ and $0 < \al < \be \le 1$. 
\end{itemize}
As a corollary we obtain that as a mapping 
\begin{equation} \label{eq:intro-flowmap2}
	L^1([0,1],C^{n,\be-}_0(\R^d,\R^d)) \to C([0,1],\cD^{n,\be-}(\R^d))
\end{equation}
the flow map is bounded for all $n\ge 1$ and continuous and $C^{0,\om}$ if $n\ge 2$ (and arbitrary $\be \in (0,1]$),  
for every slowly vanishing modulus of continuity $\om$.

In \cite{BruverisVialard14} similar results were obtained in the Sobolev case $E = H^s(\R^d,\R^d)$, for $s > d/2+1$.
In particular, it was shown that $H^s$ is ODE closed and that
\[
	\cG_{s}(\R^d) = \cD^s(\R^d)_0,
\] 
where $\cG_{s}(\R^d)$ denotes the corresponding Trouv\'e group and $\cD^s(\R^d)$ the group of orientation preserving  
diffeomorphisms of $\R^d$ which differ from the identity by a mapping in $H^s(\R^d,\R^d)$. 
The methods are quite different: thanks to the fact that $D^s(\R^d)$ is a topological group 
(cf.\ \cite{InciKappelerTopalov13}) 
Carath\'eodory's solution theory for ODEs on Banach spaces which are Bochner integrable in time is well suited 
for this setting.

The paper is structured as follows. 
We fix notation and present the main technical tools in \Cref{sec:preparation}.
We also review some results on the composition of H\"older functions essentially due to \cite{LlaveObaya99}; 
since we need slightly altered versions we give proofs but relegate them to \Cref{appendix}.
In \Cref{sec:groups} we investigate the groups $\cD^{n,\be}(\R^d)$, $\cD^{n,\be-}(\R^d)$  and $\cD^{n,\be+}(\R^d)$. 
We prove ODE closedness of $C^{n,\be}_0$ and the identities \eqref{eq:intro1} and \eqref{eq:intro2} in 
\Cref{Trouve}. 
In \Cref{sec:flows} we study the continuity of the flow maps \eqref{eq:intro-flowmap1} and 
\eqref{eq:intro-flowmap2}.

\subsection*{Acknowledgement}
We are indebted to Peter Michor who brought the Trouv\'e group to our attention and proposed the notions of 
ODE closedness and ODE hull.

\section{Definitions and preliminary results} \label{sec:preparation}

\subsection{H\"older spaces} \label{ssec:Hoelder}

	Let $k \in \N$, $\al \in (0,1]$.
	Let $E, F$ be Banach spaces and let $U \subseteq E$ be open.
	We consider the 
	\emph{global} H\"older space  
	\begin{equation*}
		C^{k,\alpha}_b (U, F):= \big\{  f \in C^k (U, F) :  \|f\|_{k,\alpha}<\infty\big\},
	\end{equation*}
	where 
	\begin{align*}
		\|f\|_{k,\alpha} &:= \max \{ \|f\|_{k}, [f]_{k,\alpha} \},
		\\
		\|f\|_{k} &:= \sup\{\|f^{(l)}(x)\|_{L_l} :x \in U,~ 0\leq l \leq k \},
		\\
		[f]_{k,\alpha} &:= \sup_{x,y \in U, \, x \ne y} \frac{\| f^{(k)}(x) -f^{(k)}(y)\|_{L_k}}{\|x-y\|^\alpha}.
	\end{align*}
	Here $f^{(l)} = d^l f: E \rightarrow L_l(E; F)$ is the Fr\'echet derivative of order $l$ and $L_l(E; F)$ 
	denotes the vector space of continuous $l$-linear mappings endowed with the operator norm $\|\cdot\|_{L_l}$.
	
	We denote by $C^{k,\alpha}_0 (E, F)$ the subspace of those mappings $f \in C^{k,\alpha}_0 (E, F)$ 
	that tend to $0$ at infinity together with 
	all their derivatives up to order $k$, i.e., for every $\ep > 0$ there is $r >0$  
	such that $\|f^{(l)}(x)\|_{L_l} \le \ep$ if $\|x\| > r$ and $0 \le l \le k$.

	All these spaces are Banach spaces.
	
	Local H\"older spaces are denoted by $C^{k,\al}$, i.e., $f \in C^{k,\al}(U,F)$ if each $x \in U$ 
	has a neighborhood $V$ in $U$ such that $f|_V \in C^{k,\al}_b(V,F)$.

	Let us recall interpolation and inclusion relations for H\"older spaces. 
	In the following $C^{n,0}_b := C^n_b$ and $\|\cdot \|_{n,0} := \|\cdot\|_n$. 
	
	\begin{lemma}[{\cite[3.1]{LlaveObaya99}}]
		\label{exp}
		Let $n \in \mathbb{N}$ and $0\leq \alpha < \beta < \gamma \leq 1 $ and set 
		$\mu := \frac{\gamma - \beta }{\gamma - \alpha}$. Then
		\begin{equation*}
			\|f\|_{n,\beta} \leq M_\alpha \|f\|_{n,\alpha}^\mu\|f\|_{n,\gamma}^{1-\mu}, 
			\quad  f \in C^{n,\gamma}_b(E, F), 
		\end{equation*}
		where $M_0 := 2$ and $M_\alpha := 1$ for $\alpha > 0$. 
	\end{lemma}

	\begin{lemma}[{\cite[3.7]{LlaveObaya99}}]
		\label{inclusion}
		Let $m,n \in \mathbb{N}$ and $\alpha, \beta \in [0,1]$ with $m+\alpha \leq n+\beta$.
		Then $C^{n,\beta}_b(E,F) \subseteq C^{m,\alpha}_b(E,F)$ and
		\begin{equation*}
			\|f\|_{m,\alpha} \leq 2 \|f\|_{n,\beta}, \quad f \in C^{n,\beta}_b(E,F).  
		\end{equation*}
	\end{lemma}
	
	
\subsection{The Bochner integral}

	Cf.\ \cite{DiestelUhl77}. 
	Let $I = [a,b] \subseteq \R$ be a closed interval (with the Lebesgue measure). Let $E$ be a Banach space.
	A measurable function $f : I \to E$ is \emph{Bochner integrable} if there is a sequence of integrable simple functions 
	$s_n : I \to E$ such that $s_n \to f$ a.e.\ (i.e., $f$ is \emph{strongly measurable}) and 
	$\int_a^b \|f-s_n\| \,dt \to 0$. In this case the \emph{Bochner integral} is defined 
	by 
	\[
		\int_a^b f \, dt = \lim_{n \to \infty}  \int_a^b s_n \,dt.
	\] 
	By the Pettis measurability theorem, $f : I \to E$ is strongly measurable if and only if it is \emph{weakly measurable} 
	(i.e., $\la \o f$ is measurable for all $\la \in E^*$) 
	and \emph{essentially separable valued} (i.e., $f(I\setminus N)$ is separable in $E$ for some null set $N$).
	A strongly measurable function $f : I \to E$ is Bochner integrable if and only if $\int_a^b\|f\| \,dt < \infty$.
	Then the triangle inequality holds:
	\[
		\Big\| \int_a^b f \,dt  \Big\| \le \int_a^b\|f\| \,dt. 
	\]  
	If $T : E \to F$ is a bounded linear operator into another Banach space $F$ then $Tf: I \to F$ is Bochner integrable
	and 
	\[
		T \int_a^b f \,dt = \int_a^b Tf \,dt.
	\]
	We will use the following version
    of the fundamental theorem of calculus. 

	\begin{lemma}
		\label{ftc}
		If $f:[a,b] \rightarrow E$ is continuous, then 
		\[\frac{d}{dt}\int_a^t f(s) \, ds = f(t),\] for all $t \in I$. 
		If $f : [a,b] \to E$ is $C^1$, then
		\begin{equation*}
			f(b)-f(a) = \int_a^b f'(s) \, ds.
		\end{equation*}
	\end{lemma}
	
	It is then straightforward to deduce a mean value inequality for $C^1$-mappings between Banach spaces.

	\subsection{Carath\'eodory type ODEs} \label{carath}

	Next we collect some results on Carath\'eodory type differential equations. 
	Those are certain ODEs on Banach spaces whose right hand side is not continuous in time. 
	We refer to \cite{AulbachWanner96} and to the appendix in 
	\cite{BruverisVialard14}.

	 Let $E$ be a Banach space, $U \subseteq E$ some open subset and $I = [t_0,t_1]$ some real interval. 
	 We say that $f:I\times U \rightarrow E$ satisfies the \textit{Carath\'eodory property} if:
	\begin{itemize}
		\item[(i)] For every $t \in I$ the mapping $f(t,\cdot):U\rightarrow E$ is continuous.
		\item[(ii)] For every $x \in U$ the mapping $f(\cdot ,x):I \rightarrow E$ is strongly measurable. 
	\end{itemize}
	
	Also the notion of solution of such an ODE is weakened: we say a continuous curve 
	$\Ph: I \rightarrow U$ is a solution of the initial value problem
	\begin{align}
		\label{bode}
		\partial_t x &= f(t,x),\quad 
		x(t_0)=x_0
	\end{align}
	if and only if $s \mapsto f(s,\Ph(s))$ is Bochner integrable and	
	\begin{align}
		\label{odesol}
		\Ph(t) = x_0 + \int_{t_0}^t f(s,\Ph(s)) \, ds, \quad  t \in I.
	\end{align}
	This already implies that $\Ph : I \rightarrow U$ is continuous. 
	It is actually absolutely continuous in the sense that there exists a Bochner integrable $\gamma: I \rightarrow E$ 
	such that $\Ph(t)= \Ph(t_0) +\int_{t_0}^t \gamma(s)\, ds$;  
	in particular, $\Ph$ is differentiable a.e.\ and $\Ph'=\gamma$ a.e.\ 
	(see \cite[Lemma 1.28]{Glockner16}).
	
	The next theorem is the central existence and uniqueness result for Carath\'eodory type differential equations; it is taken from \cite[Thm. A.2]{BruverisVialard14}.
	\begin{theorem}
		\label{cara}
		Let $I = [t_0,t_1]$ and let $f: I \times U \rightarrow E$ have the Carath\'eodory property. Let 
		$B(x_0,\varepsilon) := \{x \in E : \|x-x_0\|<\ep \} \subseteq U$. 
		In addition let $m,l$ be positive locally integrable functions defined on $I$ such that the estimates
		\begin{align*}
			\|f(t,x_1)-f(t,x_2)\| &\leq l(t) \|x_1-x_2\|\\
			\|f(t,x)\| &\leq m(t)
		\end{align*}
		are valid for almost all $t$ and all $x,x_1,x_2 \in B(x_0,\varepsilon)$. Let $\delta$ be such that
		\[
		\int_{t_0}^{t_0+\delta} m(s)\, ds < \varepsilon,
		\]
		then \eqref{bode} has a unique solution $\phi:[t_0,t_0+\delta]\rightarrow B(x_0,\varepsilon)$ in the sense of \eqref{odesol}.
	\end{theorem}
	
	If the ODE is linear, we have global existence in time:
	\begin{theorem}
		\label{lincara}
		Let $I = [t_0,t_1]$.
		Let $A:I\rightarrow L(E)$ and $b:I\rightarrow E$ be Bochner integrable. 
		Then for all $x_0 \in E$ there exists a unique solution on $I$ of 
		\begin{align*}
			\partial_t x(t) &= A(t)\cdot x(t) +b(t),\quad 
			x(t_0)=x_0
		\end{align*}
		in the sense of \eqref{odesol}.
	\end{theorem}

	\subsection{Composition in H\"older spaces} 	\label{composition}

	Let us review some regularity results for the composition in H\"older spaces due to \cite{LlaveObaya99}. 
	But in contrast to \cite{LlaveObaya99}, we need the results for mappings $F : \R^d \to \R^d$ of the form 
	$F =\Id + f$ where $f$ is in some H\"older class; note that $\Id$ is unbounded and hence not a member of 
	any $C^{n,\be}_b(\R^d,\R^d)$. For this reason it is convenient to introduce the seminorm
	\[
		[F]_n := \|F^{(n)}\|_{0} = \sup_{x \in \R^d} \|F^{(n)}(x)\|_{L_n}.
	\]
	If $F = \Id +f$ and $n \ge 1$, then 
	\begin{equation}
		[F]_n \le 1 + [f]_n.
	\end{equation}
	It is easy to adapt the proofs in \cite{LlaveObaya99} to our needs; they are 
	outlined in \Cref{appendix} for completeness' sake.

	\begin{proposition}[{\cite[4.2]{LlaveObaya99}}]
		\label{bilinear}
		Let $E,F,G,H$ be Banach spaces and $U \subseteq E$ open. Let $m \in \N$, $\al \in (0,1]$, and $b: F \times G \rightarrow H$ be a bilinear continuous mapping. Then $b_*: C^{m,\alpha}_b (U, F) \times C^{m,\alpha}_b(U,G) \rightarrow C^{m,\alpha}_b (U, H)$,
		defined by $b_*(f,g)(x):= b(f(x),g(x))$, 
		is bilinear, continuous, and $\|b_*\| \leq 2^{m+1}\|b\|$.
	\end{proposition}

	The following theorem shows stability under composition and continuity of the right translation. 
	We will denote by $f^\star$ the pull-back by $\Id + f$, i.e., $f^\star := (\Id +f)^*$.

\begin{theorem}[{\cite[6.2]{LlaveObaya99}}]  \label{comp1}
	Let $m \in \N_{\ge 1}$ and $\al \in (0,1]$.	
	Let $f \in C^{m,\alpha}_b(\mathbb{R}^d,\mathbb{R}^d)$ and $g \in C^{m,\al}_b (\mathbb{R}^d, G)$ for some Banach space $G$. 
	Then $g\circ (\Id+f) \in C^{m,\al}_b(\mathbb{R}^d,G)$ and there exists a constant $M = M(m) \ge 1$  
	such that
			\begin{equation} \label{eq:comp1}
				\|g\circ(\Id+f)\|_{m,\al} \leq M \|g\|_{m,\al} (1+\|f\|_{m,\al})^{m+\al}.
			\end{equation}
	In particular, for every fixed $f \in C^{m,\alpha}_b(\mathbb{R}^d,\mathbb{R}^d)$, the linear mapping 
	\begin{align*}
			f^\star: C^{m,\al}_b(\R^d,G) \rightarrow C^{m,\al}_b(\R^d,G),\quad g \mapsto f^\star(g) := g \circ(\Id+f)
		\end{align*}
		is continuous.	
\end{theorem}

Continuity of the left translation is the content of the following theorem.
We denote by 
$B^{m,\alpha}(f,\delta) := \{g \in C^{m,\al}_b: \|f-g\|_{m,\al} < \de\}$ the open ball with radius $\de$ centered at $f$. 
By $g_\star$ we mean the push-forward by $g$ precomposed with translation by $\Id$, i.e., $g_\star := g_* \o (\Id + \,\cdot\,)$.

	\begin{theorem}[{\cite[6.2]{LlaveObaya99}}]
		\label{6.2}
		Let $m \in \N_{\ge 1}$ and $\al,\be \in (0,1]$, $\al < \be$.
		Let  $g \in C^{m,\beta}_b (\mathbb{R}^d, G)$ where $G$ is some Banach space. 
		Then, for every $f_0 \in C^{m,\alpha}_b(\R^d,\R^d)$, $R>0$,  and 
		$f_1, f_2 \in B^{m,\alpha}(f_0,R)$,
			\begin{equation} \label{eq:6.2}
				\|g_\star(f_1)-g_\star(f_2)\|_{m,\al} \le M\|g\|_{m,\beta} \|f_1-f_2\|_{m,\al}^{\beta - \al},
			\end{equation}
		where $M = M(m, \|f_0\|_{m,\al},R)$. 	
		In particular,  
		\begin{align*}
			g_\star: C^{m,\alpha}_b(\mathbb{R}^d,\mathbb{R}^d) \rightarrow C^{m,\al}_b(\mathbb{R}^d, G), 
			\quad f \mapsto g_\star(f):= g \circ ( \Id + f)
		\end{align*}
		is continuous.
	\end{theorem}

	It follows that composition is even jointly continuous.

	\begin{corollary}
		\label{jcont}
		Let $m \in \N_{\ge 1}$, $0< \alpha < \beta \leq 1$ and $G$ some Banach space. 
		Then, for all $f_0 \in C^{m,\alpha}_b(\R^d,\R^d)$, $g_0 \in C^{m,\be}_b(\R^d,G)$, $R>0$,   
		$f_1, f_2 \in B^{m,\alpha}(f_0,R)$, and $g_1,g_2 \in B^{m,\be}(g_0,R)$,
			\begin{equation} \label{eq:jcont}
				\|g_1 \o (\Id + f_1)-g_2\o  (\Id + f_2)\|_{m,\al} 
				\le M\big(\|g_1 - g_2\|_{m,\al} + \|f_1-f_2\|_{m,\al}\big)^{\beta - \al},
			\end{equation}
		where $M = M(m, \|f_0\|_{m,\al},\|g_0\|_{m,\be},R)$.
		In particular,
		\begin{align*}
			\comp:~ &C^{m,\beta}_b(\R^d, G)\times C^{m,\alpha}_b(\R^d,\R^d) \rightarrow C^{m,\alpha}_b(\R^d,G), 
			\quad (g,f)	\mapsto g\circ(\Id+f)	
		\end{align*}
		is continuous.
	\end{corollary}

We will also need the following result on $C^1$ left translations.

	\begin{theorem}[{\cite[6.7]{LlaveObaya99}}]
	\label{6.7}
		Let $m \in \N_{\ge 1}$ and $\al,\be \in (0,1]$, $\al < \be$.
		Let  $g \in C^{m+1,\beta}_b (\mathbb{R}^d, G)$ where $G$ is some Banach space.
		Then $g_\star: C^{m,\alpha}_b(\mathbb{R}^d,\mathbb{R}^d) \rightarrow C^{m,\al}_b(\mathbb{R}^d, G)$
		is continuously differentiable.
	\end{theorem}

Together with \Cref{ftc}, \Cref{6.7} implies Lipschitz continuity of the left translation in the following cases; 
but see also \Cref{lefttranslationLip} below.

\begin{corollary} \label{Lipschitz}
	Let $m \in \N_{\ge 1}$ and $0 <\al < \be \le 1$.
	Let $g \in C^{m+1,\beta}_b(\mathbb{R}^d,G)$. 
	Then 
	$g_\star: C^{m,\alpha}_b(\mathbb{R}^d,\mathbb{R}^d) \rightarrow C^{m,\alpha}_b(\mathbb{R}^d,G)$
	satisfies for all $f_1,f_2 \in C^{m,\alpha}_b(\mathbb{R}^d,\mathbb{R}^d)$, 
	\begin{align*}
	\|g_\star( f_1) - g_\star(f_2)\|_{m,\alpha}  
	\leq  M \|g\|_{m+1,\beta} (1 + \max_{i=1,2}\|f_i\|_{m,\alpha})^{m+1}  \|f_1-f_2\|_{m,\alpha}.
	\end{align*} 
\end{corollary}

\begin{remark}
	Let us stress the fact that left translation ceases to be continuous, resp.\ differentiable, if in 
	\Cref{6.2}, resp.\ \Cref{6.7}, $g$ is merely of class $C^{m,\al}_b$, resp.\ $C^{m+1,\al}_b$; see \cite{LlaveObaya99} 
	and also \Cref{disc}.		
\end{remark}

	We shall make frequent use of the Fa\`a di Bruno formula for Banach spaces:
		Let $E,F,G$ be Banach spaces, let $f: E \supseteq U  \rightarrow F$ and $g: F \supseteq V \rightarrow G$ 
		be $k$ times Fr\'echet differentiable, and assume $f(U)\subseteq V$. Then $g\circ f:U \rightarrow G$ is $k$ 
		times Fr\'echet differentiable, and for all $x \in U$,		
		\begin{align}
			\label{eq:FaadiBruno}
				d^k(g\circ f)(x) = \on{sym}
				\sum_{l=1}^{k} \sum_{ \ga \in \Ga(l,k)} c_{\ga} g^{(l)}(f(x))\left(f^{(\ga_1)}(x), \dots, f^{(\ga_l)}(x)\right),
		\end{align}
		where $\Ga(l,k) := \{\ga \in \N_{>0}^l : |\ga| = k\}$,
		$c_{\ga} := \frac{k!}{l! \ga!}$, and $\on{sym}$ denotes symmetrization of multilinear mappings.\par
		Fa\`a di Bruno's formula applied to a function $h:U \rightarrow H$ of the form $h(x)=b(f(x),g(x))$, where $f,g$ are $k$ times Fr\'echet differentiable functions defined on a common domain $U \subseteq E$ and $b:F\times G \rightarrow H$ is a continuous bilinear map, gives 
		\begin{align}
		\label{eq:productrule}
			d^kh(x) = \on{sym}
			\sum_{l=0}^{k} \binom{k}{l} b(f^{(l)}(x),g^{(k-l)}(x)).
		\end{align}
		This formula is of particular use when $h(x)=dg(x)(f(x))$, where $f,g:\R^d \rightarrow \R^d$, i.e. the bilinear map takes the form $b:L(\R^d,\R^d)\times \R^d \rightarrow \R^d$, $(l,x)\mapsto l(x)$.
	
	\begin{remark}
		\label{compcl}
		Fa\`a di Bruno's formula \eqref{eq:FaadiBruno} implies that for $f:\R^d \rightarrow \R^d$ and $g:\R^d \rightarrow G$ 
		both in $C^k_0$, we have $g \circ(\Id+f) \in C^{k}_0(\R^d,G)$. 
		So the stated regularity results for the composition hold as well for $C^{m,\al}_b$, etc., replaced by $C^{m,\al}_0$, 
		etc.
	\end{remark}

\subsection{Convenient calculus} \label{ssec:reviewconvenient}

Occasionally, we shall use some tools from \emph{convenient calculus} which extends differential calculus 
beyond Banach spaces; the main reference is \cite{KM97}, see also \cite{FK88} and the three appendices in 
\cite{KMRc}.   
Let us briefly describe the concepts and results we will need.

Let $E$ be a locally convex vector space. A curve $c: \mathbb R\to E$ is called $C^\infty$ if all derivatives exist and are 
continuous. It can be 
shown that the set $C^\infty(\mathbb R,E)$ of $C^\infty$-curves in $E$ does not depend on the locally convex 
topology of $E$, only on its associated bornology.

The \emph{$c^\infty$-topology} on $E$ is the final topology with respect to $C^\infty(\R,E)$; 
equivalently it is the final topology with respect to all Lipschitz curves or all Mackey-convergent sequences in $E$.
In general 
the $c^\infty$-topology is finer 
than the given locally convex topology, and it is not a vector space 
topology; for Fr\'echet spaces the topologies coincide.

A locally convex vector space 
$E$ is said to be a \emph{convenient 
vector space} if it is Mackey-complete; equivalently, a 
curve $c:\mathbb R\to E$ is $C^\infty$ if and only if $\la\o c$ is 
$C^\infty$ for all continuous (equivalently bounded) linear functionals $\la$ on $E$.

Let $E$, $F$, and $G$ be convenient vector spaces, 
and let $U\subseteq E$ be $c^\infty$-open. 
A mapping $f:U\to F$ is called  
$C^\infty$, if $f\o c\in C^\infty(\mathbb R,F)$ for all 
$c\in C^\infty(\mathbb R,U)$.
For mappings on Fr\'echet spaces this notion of smoothness 
coincides with all other reasonable definitions.
Multilinear mappings are $C^\infty$ if and only if they are 
bounded. The space $C^\infty(U,F)$ with the initial structure with respect to all mappings $f \mapsto \la \o f \o c$, 
$c \in C^\infty(\R,E)$ and $\la \in E^*$,
is again convenient. The exponential law holds: For $c^\infty$-open $V\subseteq F$, 
\[
C^\infty(U,C^\infty(V,G)) \cong C^\infty(U\times V, G)
\]
is a linear diffeomorphism of convenient vector spaces. 
A linear mapping $f:E\to C^\infty(V,G)$ is $C^\infty$ (bounded) if 
and only if $\ev_v \o f : E \to G$ is $C^\infty$ for all $v \in V$.

There are, however, $C^\infty$-mappings which are not continuous with respect to the underlying 
locally convex topology; clearly they are continuous for the $c^\infty$-topology.

Beside the class $C^\infty$ due to \cite{Frolicher81}, \cite{Kriegl82}, \cite{Kriegl83}, 
convenient calculus was developed for 
the holomorphic class \cite{KrieglNel85}, 
the real analytic class \cite{KrieglMichor90}, 
and all reasonable ultradifferentiable classes \cite{KMRc}, \cite{KMRq}, \cite{KMRu}, \cite{Schindl14a}.  

For the classes $C^{k,\al}$ ($k \in \N$, $\al \in (0,1]$) it was established in a weaker sense (without general exponential law) 
by \cite{FroelicherGisinKriegl83} (for $\al =1$) 
and by \cite{FaureFrolicher89}, \cite{Faure91}.  
Let $E$, $F$ be convenient vector spaces, 
and let $U\subseteq E$ be $c^\infty$-open. 
A curve $c : \R \to F$ is \emph{locally $\al$-H\"older continuous}, we write $c \in C^{0,\al}(\R,F)$, if 
for each bounded interval $I \subseteq \R$,
\[
	\Big\{\frac{c(t) - c(s)}{|t-s|^\al} : t,s \in I,\, t \ne s \Big\}
\]
is bounded in $F$. A curve $c : \R \to F$ is $C^{k,\al}$, i.e., $c \in C^{k,\al}(\R,F)$, if 
all 
derivatives up to order $k$ exist and are locally $\al$-H\"older continuous.
A mapping $f:U\to F$ between convenient vector spaces is called  
$C^{k,\al}$, if $f\o c\in C^{k,\al}(\mathbb R,F)$ for all 
$c\in C^\infty(\mathbb R,U)$. 
If $E$ and $F$
are Banach spaces, then $f$ is $C^{0,\al}$ in this senses if and only if it is in the sense of 
\Cref{ssec:Hoelder}, i.e., $\|f(x) - f(x)\|/\|x-y\|^\al$ is locally bounded; 
see \cite{Faure91}, \cite[12.7]{KM97}, or \cite[Lemma]{KMR}, and note that 
this is a special case of 
\Cref{lem:C0om} below. 

\subsection{An application of convenient calculus}

We finish this section with a result which is not contained in \cite{LlaveObaya99}: 
if $\al = \be$ in \Cref{6.7}, resp.\ \Cref{Lipschitz}, the left translation $g_\star$ is still locally Lipschitz. 
Of course, \Cref{compcl} applies to this theorem as well. In contrast to \Cref{Lipschitz}, we do not get an 
explicit bound for the Lipschitz constant.

	\begin{theorem} \label{lefttranslationLip}
		Let $m \in \N_{\ge 1}$ and $0 <\al  \le 1$.
		Let $g \in C^{m+1,\al}_b(\mathbb{R}^d,\mathbb{R}^d)$.  
		Then 
		$g_\star: C^{m,\alpha}_b(\mathbb{R}^d,\mathbb{R}^d) \rightarrow C^{m,\alpha}_b(\mathbb{R}^d,\mathbb{R}^d)$
		is locally Lipschitz. 
	\end{theorem}

\begin{proof}
	It suffices to check that $g_\star$ maps $C^\infty$-curves to $C^{0,1}$-curves; cf.\ \Cref{ssec:reviewconvenient}. 
	That $t \mapsto  f(t, \cdot )$ is $C^\infty$ in $C^{m,\alpha}_b(\mathbb{R}^d,\mathbb{R}^d)$ 
	means, by \cite[4.1.19]{FK88}, that, for all $k \in \N$,  $\|\p_t^k  f(t, \cdot )\|_{m,\al}$ is locally bounded in $t$.	
	
	Let $h(t,x) := g(x+f(t,x))$. Then, if $F := \Id +f$,  
	\[
		h(t,x) - h(s,x) = \int_s^t \p_\ta h(\ta,x)\, d \ta = \int_s^t dg(F(\ta,x)) \p_\ta f(\ta,x)\, d \ta 
	\]
	and, by \eqref{eq:productrule},
	\begin{align*}
		d^k_xh(t,x) - d^k_xh(s,x) 
		&= \int_s^t d_x^k \big( dg(F(\ta,x)) \p_\ta f(\ta,x) \big) \, d \ta
		\\
		&= \on{sym}\sum_{j=0}^k \binom{k}{j} \int_s^t d_x^{j} \big( dg(F(\ta,x)) \big) \p_\ta d_x^{k-j} f(\ta,x)  \, d \ta.
	\end{align*}
	With Fa\`a di Bruno's formula \eqref{eq:FaadiBruno},
	\begin{align*}
		d_x^{j} \big( dg(F(\ta,x)) \big)  
		&=  \on{sym} \sum_{l=1}^{j} \sum_{\ga \in \Ga(l,j)}
				 c_{\ga} g^{(l+1)}(F(\ta,x))\big(d_x^{\ga_1}F(\ta,x), \dots, d_x^{\ga_l} F(\ta,x), \mathbb 1\big)
	\end{align*}
	it is easy to see that $t \mapsto h(t, \cdot)$ is locally Lipschitz into $C^{m}_b(\R^d,\R^d)$.

	It remains to prove that $t \mapsto h(t, \cdot)$ is locally Lipschitz into $C^{m,\al}_b(\R^d,\R^d)$. To this end we have to show 
	that 
	\[
		\frac{[h(t,\cdot)-h(s,\cdot)]_{m,\al}}{t-s} 
	\]
	is locally bounded, i.e., for each bounded interval $I$, the set
	\begin{align*}
		\Big\{
		\frac{d^{m}_x h(t,x) - d^{m}_x h(t,y) - d^{m}_x h(s,x) + d^{m}_x h(s,y) }{\|x-y\|^\al|t-s|} : 
		x \ne y \in \R^n, s \ne t \in I
		\Big\}
	\end{align*}
	must be bounded. 
	Without loss of generality we can assume that $\|x- y\| \le 1$ and thus $\|x-y\| \le \|x-y\|^\al$; 
	if $\|x- y\| \ge 1$ then the result follows from the fact that $t \mapsto h(t, \cdot)$ is locally Lipschitz 
	into $C^{m}_b(\R^d,\R^d)$.
	Let us define
	\[
		A^{\ga,i} = A^{\ga,i}(x,y) :=  \big(d_x^{\ga_1} F(\ta,x), \dots, d_x^{\ga_i}F(\ta,x), 
		d_x^{\ga_{i+1}}F(\ta,y),
		\dots, d_x^{\ga_l} F(\ta,y)\big)
	\] 
	and 
	\begin{align*}
		B^{\ga,h} := \begin{cases}
			\big(A^{\ga,l},\p_t d_x^{m-j} f(\ta,x)\big) & \text{ if } h= l+1, 
			\\
			\big(A^{\ga,h},\p_t d_x^{m-j} f(\ta,y)\big) & \text{ if } h \le l.
			\\
		\end{cases}	
	\end{align*}
	Then
	\begin{align*}
		\MoveEqLeft
		g^{(l+1)}(F(\ta,x))\big(B^{\ga,l+1}\big) -
		g^{(l+1)}(F(\ta,y))\big(B^{\ga,0}\big) 
		\\ 
		&= g^{(l+1)}(F(\ta,x))(B^{\gamma,l+1}) - g^{(l+1)}(F(\ta,y))(B^{\gamma,l+1})\\
				& \quad + 
				\sum_{h=1}^{l+1}  g^{(l+1)}(F(\ta,y))(B^{\gamma,h}) 
				- g^{(l+1)}(F(\ta,y))(B^{\gamma,h-1}).
	\end{align*}
	For the first summand		
\begin{align*}
	\MoveEqLeft
	\big\| g^{(l+1)}(F(\ta,x))(B^{\gamma,l+1}) - g^{(l+1)}(F(\ta,y))(B^{\gamma,l+1})\big\|_{L_{m}}
	\\
	&\le \big\| g^{(l+1)}(F(\ta,x)) - g^{(l+1)}(F(\ta,y))\big\|_{L_{l+1}} (1 + \|f(\ta,\cdot)\|_{m})^{m} 
	\|\p_t f(\ta,\cdot)\|_{m-j}
	\\
	&\le \begin{cases}
		\| g\|_{m+1} [F(\ta,\cdot)]_{1} \| x- y\|  (1 + \|f(\ta,\cdot)\|_{m})^{m} 
		\|\p_t f(\ta,\cdot)\|_{m} & \text{ if } l < m, \\ 
		\| g\|_{m+1,\al} [F(\ta,\cdot)]_{1}^\al \| x- y\|^\al  (1 + \|f(\ta,\cdot)\|_{m})^{m} 
		\|\p_t f(\ta,\cdot)\|_{m} & \text{ if } l = m.
	\end{cases} 
\end{align*}
For the other summands we observe that, by multilinearity,
\begin{align*}
			g^{(l+1)} (F(\ta,y))(B^{\gamma,h}) 
				- g^{(l+1)}(F(\ta,y))(B^{\gamma,h-1}) = g^{(l+1)}(F(\ta,y)) \big( \sharp \big),
		\end{align*}
where 
\begin{align*}
			\sharp = \big(  \dots, d_x^{\ga_{h-1}} F(\ta,x),
			 d_x^{\ga_{h}}F(\ta,x) - d_x^{\ga_{h}}F(\ta,y), d_x^{\ga_{h+1}} F(\ta,y),\dots \big).
		\end{align*}
		Hence, if $h \le l$,
		\begin{align*}
			\MoveEqLeft
			\big\|g^{(l+1)}(F(\ta,y))(B^{\gamma,h}) 
				- g^{(l+1)}(F(\ta,y))(B^{\gamma,h-1})\big\|_{L_{m+1}}\\
			&\le \| g\|_{m+1} (1 + \|f(\ta,\cdot)\|_{m})^{m-1}
			   \|f(\ta,\cdot)\|_{m,\al}\|x-y\|^\al \|\p_t f(\ta,\cdot)\|_{m},
		\end{align*}
		and, if $h= l+1$,
		\begin{align*}
			\MoveEqLeft
			\big\|g^{(l+1)}(F(\ta,y))(B^{\gamma,h}) 
				- g^{(l+1)}(F(\ta,y))(B^{\gamma,h-1})\big\|_{L_{m+1}}\\
			&\le \| g\|_{m+1} (1 + \|f(\ta,\cdot)\|_{m})^{m}
			   \|\p_t f(\ta,\cdot)\|_{m,\al}\|x-y\|^\al.
		\end{align*}
The theorem follows.
\end{proof}

	\section{Groups of H\"older diffeomorphisms} \label{sec:groups}

	\subsection{The (non-topological) group \texorpdfstring{$\mathcal{D}^{n,\beta}(\R^d)$}{Dnbeta}}

	Let $n \in \N_{\ge 1}$ and $\be \in (0,1]$.
	Let us define the set of orientation preserving diffeomorphisms of $\R^d$ 
	which differ from the identity by a $C^{n,\be}_0$-mapping:
	\begin{equation}
	\label{diffgroup}
		\mathcal{D}^{n,\beta}(\mathbb{R}^d):= \big\{ \Ph \in \Id+C^{n,\beta}_0(\mathbb{R}^d,\mathbb{R}^d): 
		\det d\Phi(x) > 0 ~\forall x \in \mathbb{R}^d \big\}.
	\end{equation}
	We will show that $\cD^{n,\be}(\R^d)$ is a group (with respect to composition).

	We endow $\mathcal{D}^{n,\beta}(\R^d)$ with the topology given by the metric
	\[
	d(\Ph_1,\Ph_2):= \|\Ph_1-\Ph_2\|_{n,\beta}
	\]
	and denote by $B^{n,\beta}(\Ph,r)$ the open ball of radius $r$ and center $\Ph$ in $\mathcal{D}^{n,\beta}(\R^d)$. 
	We use the same notation for balls in $C^{n,\beta}_0(\R^d,\R^d)$ which causes no problems since 
	$\Id \not \in C^{n,\beta}_0(\R^d,\R^d)$.

	Since the determinant is multiplicative, it is an easy consequence of \Cref{comp1} that $\mathcal{D}^{n,\beta}(\R^d)$ 
	is a monoid with respect to composition.

	\begin{lemma}
		\label{ndiff}
		$\mathcal{D}^{n,\beta}(\R^d)$ consists of $C^n$-diffeomorphisms of $\R^d$. 
		The first $n$ derivatives of the inverse of an element of $\mathcal{D}^{n,\beta}(\R^d)$ are again globally bounded.
	\end{lemma}
	
	\begin{proof}
		Let $\Phi = \Id+\ph \in \mathcal{D}^{n,\beta}(\R^d)$. First we have to make sure that $\Ph$ is bijective. 
		This is an immediate consequence of \cite[Cor. 4.3]{Palais59}, which states that a $C^1$ mapping converging 
		to infinity at infinity with non-vanishing jacobian determinant is already a $C^1$ diffeomorphism. 
		The inverse mapping theorem shows that $\Phi^{-1}$ is actually $C^n$. 
		Boundedness of the first $n$ derivatives of $\Ph^{-1} - \Id$ follows as in \cite[p.~7,8]{MichorMumford13}. 
	\end{proof}

	\begin{lemma}[{\cite[p.~7]{MichorMumford13}}] \label{help0}
		The operator norm of an invertible linear operator $A: \mathbb{R}^d \rightarrow \mathbb{R}^d$ satisfies
		$\|A^{-1}\| \leq |\det A|^{-1} \|A\|^{d-1}$.
	\end{lemma}

	\begin{lemma}
		\label{help}
		Let $\Phi_0 = \Id + \ph_0 \in  \mathcal{D}^{n,\beta}(\R^d)$. Then: 
		\begin{enumerate}
			\item $\varepsilon := \inf_{x\in \R^d} \det d\Ph_0(x) >0$.
			\item There is $\delta > 0$ such that   
				$\inf_{x\in \R^d} \det d\Ph(x) \geq \varepsilon/2$ for all $\Ph \in \Id+ B^{n,\be}(\ph_0,\de)$.
			\item There are $\delta, C > 0$ such that  
			$\sup_{x \in \R^d}\|d \Ph^{-1}(x)\| \le C$ for all $\Ph \in B^{n,\be}(\Ph_0,\de)$.
		\end{enumerate}
	\end{lemma}
	
	\begin{proof} 
		(1)	
		Observe that $d \Ph_0(x) \rightarrow \mathbb{1}$ as $\|x\| \to \infty$. 
		Thus $\det d\Ph_0(x) \rightarrow 1$ as $\|x\| \rightarrow \infty$, 
		which implies $\varepsilon := \inf_{x\in \R^d} \det d\Ph_0(x) >0$.

		(2)	This follows from the fact that 
		the determinant is uniformly continuous on each ball in the space of $d \times d$ matrices.

		(3)
		Let $\de>0$ be as in (2).
		Then, for all $\Ph \in B^{n,\be}(\Ph_0,\de)$,
		\begin{align*}
		 \|	d\Ph^{-1}(\Ph(x))\|=\|(d\Ph(x))^{-1}\|\leq &\frac{\|d\Ph(x)\|^{d-1}}{|\det d\Ph(x)|} 
		 \le \frac2{\ep} (\|\Ph_0\|_{n,\be}+\de)^{d-1},
		\end{align*}
		by \Cref{help0}. Since $\Ph$ is bijective, the proof is complete.
	\end{proof}

		\Cref{help} shows that $\mathcal{D}^{n,\beta}(\R^d) - \Id$ is an open subset of $C^{n,\beta}_0(\R^d,\R^d)$. 
		Thus, for $\Phi_0=\Id + \ph_0 \in \mathcal{D}^{n,\beta}(\R^d)$ and for sufficiently small $r>0$,
		\[
		B^{n,\beta}(\Phi_0,r) = \Id + B^{n,\beta}(\ph_0,r).
		\]
		We interpret $\mathcal{D}^{n,\beta}(\R^d)$ as a Banach manifold modelled on $C^{n,\beta}_0(\R^d,\R^d)$ with global chart 
		$\Phi \mapsto \Phi - \Id$.

	\begin{theorem} \label{thm:group}
		Let $n \in \N_{\ge 1}$ and $\be \in (0,1]$.
		Then $\cD^{n,\be}(\R^d)$ is a group. 
		In general, left translations are discontinuous.
	\end{theorem}

	The theorem will follow from \Cref{disc} and \Cref{invcl}.

	 \begin{lemma}
	 	\label{disc}
	 	In general, left translations in $\cD^{n,\beta}(\R^d)$ are discontinuous.
	 \end{lemma}
	 
	 \begin{proof}
	 	The construction is taken from \cite[6.4]{LlaveObaya99}.
	 	We prove the claim in the case $d = 1$. 
	 	Let $\chi \in C^\infty_c(\R)$ be $1$ on $[-1,1]$, and set $\ps(x):= x^n |x|^\beta \ch(x)$. 
	 	Then $\ps \in C^{n,\beta}_0(\mathbb{R},\mathbb{R})$. In addition, 
	 	let $\Ph_k(x):= x + \ch(x)/k$. Since $\mathcal{D}^{n,\beta}(\mathbb{R}) - \Id$ is open, 
	 	we have $\Ph_k \in \mathcal{D}^{n,\beta}(\mathbb{R})$, for sufficiently large $k$, 
	 	and $\Ph_k \rightarrow \Id$ in $\cD^{n,\be}(\R)$ 
	 	as $k \rightarrow \infty$.  
	 	It is easy to see that, for $|x|<1$,
	 	\begin{equation} \label{eq:constant}
	 		\ps^{(n)}(x) = (n+\beta)\cdots (1+\beta) |x|^\beta =: C_{n,\be} |x|^\beta. 
	 	\end{equation}
	 	Thus, for large $k$,
	 	\begin{equation*}
	 		(\ps \circ \Ph_k)^{(n)}\Big(-\frac{1}{k}\Big) = C_{n,\be}  
	 		\Big|-\frac{1}{k} + \frac{1}{k}\Big|^\beta = 0,
	 	\end{equation*}
	 	and
	 	\begin{equation*}
	 		(\ps \circ \Ph_k)^{(n)}(0) = \frac{C_{n,\be}}{k^\beta}.
	 	\end{equation*}
	 	Hence
	 	\begin{align*}
	 		\MoveEqLeft
	 		\big((\ps \circ \Ph_k)^{(n)} - (\ps \circ \Id)^{(n)}\big)\Big(-\frac{1}{k}\Big) 
	 		- \big((\ps \circ \Ph_k)^{(n)} - (\ps \circ \Id)^{(n)}\big)(0) = - \frac{2C_{n,\be}}{k^\beta},
	 	\end{align*}
	 	which immediately gives $\|\ps\circ \Ph_k - \ps \circ \Id\|_{n,\beta} \ge 2C_{n,\be}$. 
	 	Since $\mathcal{D}^{n,\beta}(\mathbb{R}) - \Id$ is open, 
	 	there is some small $r>0$ such that $\Id+r\ps \in \mathcal{D}^{n,\beta}(\mathbb{R})$.  
	 \end{proof}
	
	The next proposition completes the proof of \Cref{thm:group}.

	\begin{proposition}
		\label{invcl}
		$\mathcal{D}^{n,\beta}(\R^d)$ is closed under inversion. The chart representation
		\begin{align*}
			\inv_c : (\mathcal{D}^{n,\beta}(\R^d)-\Id) \rightarrow (\mathcal{D}^{n,\beta}(\R^d)-\Id),
			\quad \ph \mapsto (\Id+\ph)^{-1} -\Id
		\end{align*}
		is locally bounded.
	\end{proposition}
	
	\begin{proof}
		For $\Ph = \Id+\ph \in \mathcal{D}^{n,\beta}(\R^d)$ and $\Ph^{-1}=:\Id+\tau$ we have 
		$(\Id+\tau)\circ(\Id+\ph) = \Id$, i.e., 
		\begin{align} \label{eq:equation}
			\tau(x+\ph(x)) = -\ph(x), \quad x \in \R^d.
		\end{align}
		It follows that $\det d\Phi^{-1}(x) > 0$ for all $x$ and that $\tau \circ (\Id+\ph) \in C^{n,\beta}_0(\R^d,\R^d)$. 
		By \Cref{ndiff}, $\Phi^{-1}$ is $n$-times differentiable with globally bounded derivatives.

		Let $\Phi_0 = \Id+\ph_0 \in \mathcal{D}^{n,\beta}(\R^d)$ and $\Phi_0^{-1}=:\Id+\tau_0$.
		Choose $\delta > 0$ such that $B^{n,\beta}(\ph_0,\delta) \subseteq (\mathcal{D}^{n,\beta}(\R^d)-\Id)$ 
		(recall that $\mathcal{D}^{n,\beta}(\R^d)-\Id$ is open) and such that the conclusion of \Cref{help} holds.

		\begin{claim} \label{claim1}
			$\inv_c(B^{n,\beta}(\ph_0,\delta))$ is bounded in $C^{n}_0(\R^d,\R^d)$.
		\end{claim}

		By \Cref{ndiff}, we know that $\inv_c$ maps into $C^{n}_b(\R^d,\R^d)$. An inspection of Fa\`a di Bruno's formula 
		\eqref{eq:FaadiBruno} 
		shows that it actually maps into $C^{n}_0(\R^d,\R^d)$. 
		Let $\ph \in B^{n,\beta}(\ph_0,\delta)$ and $\tau = \inv_c(\ph) \in C^{n}_0(\R^d,\R^d)$ so that 
		\eqref{eq:equation} implies
		\[
		\|\tau(x + \ph(x))\|\leq \|\ph(x)\|\leq  \|\ph-\ph_0\|_{n,\beta}+\|\ph_0\|_{n,\beta} \leq \delta + \|\ph_0\|_{n,\beta}
		\]
		for all $x$. Since $\Id+\ph$ is bijective, this gives 
		\[
		\|\inv_c(\ph)\|_{0} = \|\tau\|_{0} \leq \delta + \|\ph_0\|_{n,\beta}, 
		\quad \ph \in B^{n,\beta}(\ph_0,\delta).
		\]

		We prove by induction on $k$ that 
		for all $k \le n$ there are constants $D_k = D_k(\ph_0,\de)$ such that 
		\begin{equation} \label{eq:inductiveassumption}
			\|\tau\|_{k} = \|\inv_c(\ph)\|_{k} \leq D_{k}, \quad \ph \in B^{n,\beta}(\ph_0,\delta).
		\end{equation}
		By Fa\`a di Bruno's formula \eqref{eq:FaadiBruno}, 
		\begin{align}
				\label{faa}
				 d^k(\tau\circ \Ph)(x) &= \tau^{(k)}(\Ph(x))(d \Ph(x), \dots, d \Ph(x) ) \notag \\
				&\quad + \on{sym} \sum_{l=1}^{k-1} \sum_{\ga \in \Ga(l,k)}
				 c_{\ga} \tau^{(l)}(\Ph(x))\big(\Ph^{(\ga_1)}(x), \dots, \Ph^{(\ga_l)}(x)\big).
		\end{align}
		By the induction hypothesis, for $\ph \in B^{n,\beta}(\ph_0,\delta)$ and $l \leq k-1$,
		\begin{align*}
			\MoveEqLeft 
			\big\|\tau^{(l)}(\Ph(x))\big(\Ph^{(\ga_1)}(x), \dots, \Ph^{(\ga_l)}(x)\big)\big\|_{L_k} \\
			&\leq~  \|\tau^{(l)}(\Ph(x))\|_{L_l}\|\Ph^{(\ga_1)}(x)\|_{L_{\ga_1}}\cdots\|\Ph^{(\ga_l)}(x)\|_{L_{\ga_l}}\\
			&\leq  D_{k-1} (1+\|\ph\|_{n,\beta})^k\\
			&\leq  D_{k-1} (1+\delta + \|\ph_0\|_{n,\beta})^k.
		\end{align*}  
		In addition, 
		$\|d^k(\tau \circ \Ph)(x)\|_{L_k}\leq \|\ph\|_{n,\beta}\leq \delta + \|\ph_0\|_{n,\beta}$,
		by \eqref{eq:equation}.
		It follows that there is some constant $D_k = D_k(\ph_0,\de)$ such that 
		\begin{equation} \label{eq:inverse}
			\|\tau^{(k)}(\Ph(x))(d \Ph(x), \dots, d \Ph(x) )\|_{L_k} \le D_k, 
			\quad \ph \in B^{n,\beta}(\ph_0,\delta),\, x \in \R^d.
		\end{equation}
		Since
		\begin{align*}
			\|\tau^{(k)}(\Ph(x))\|_{L_k}
			&\leq \|\tau^{(k)}(\Ph(x))(d \Ph(x), \dots, d \Ph(x) )\|_{L_k}
			\|(d\Ph(x))^{-1}\|_{L_1}^k,
		\end{align*}
		\eqref{eq:inverse} and \Cref{help} imply \eqref{eq:inductiveassumption} and hence Claim \ref{claim1}.

		\begin{claim} \label{claim3}
			$\inv_c(B^{n,\beta}(\ph_0,\delta))$ is bounded in $C^{n,\be}_0(\R^d,\R^d)$.
		\end{claim}

		Observe that, since $\Ph$ is a bijection of $\R^d$,
		\begin{align*}
			[\ta]_{n,\be}  
			&= \sup_{x\ne y}  \frac{\|d^n \tau(\Ph(x)) - d^n\tau(\Ph(y))\|_{L_n}}{\|x - y\|^\beta}
			\frac{\|x-y\|^\beta}{\|\Ph(x) - \Ph(y)\|^\beta}\\
			&\leq \sup_{x\ne y}  \frac{\|d^n \tau(\Ph(x)) - d^n\tau(\Ph(y))\|_{L_n}}{\|x - y\|^\beta} 
			\Big(\sup_{x\ne y}   \frac{\|x-y\|}{\|\Ph(x) - \Ph(y)\|}\Big)^\beta\\
			&=\sup_{x\ne y}  \frac{\|d^n \tau(\Ph(x)) - d^n\tau(\Ph(y))\|_{L_n}}{\|x - y\|^\beta} 
			\on{Lip}(\Ph^{-1})^\be\\
			&\le \sup_{x\ne y}  \frac{\|d^n \tau(\Ph(x)) - d^n\tau(\Ph(y))\|_{L_n}}{\|x - y\|^\beta} 
			(1+D_1)^\be,  	
		\end{align*}
		for all $\ph \in B^{n,\be}(\ph_0,\de)$,
		by \eqref{eq:inductiveassumption}.		
		For $k \leq n$, let 
		\[
			A^k=A^k(x,y) := (\underbrace{d\Ph(x), \dots, d\Ph(x)}_{k\text{-times}},
			 \underbrace{d\Ph(y), \dots, d\Ph(y)}_{(n-k)\text{-times}}).
		\]
		Then 
		\begin{align}
			\begin{split}
				\label{unicon}
				\MoveEqLeft
				\|d^n\tau(\Ph(y))(A^k) - d^n\tau(\Ph(y))(A^{k-1})\|_{L_n}\\
				&\leq \|d^n\tau(\Ph(y))\|_{L_n}\|d\Ph(x)\|_{L_1}^{k-1}\|d\Ph(y)\|_{L_1}^{n-k} \|d\ph(x)-d\ph(y)\|_{L_1}\\
				&\leq \|\tau\|_{n} (1+\|\ph\|_{1})^{n-1} 2 \|\ph\|_{n,\beta}\|x-y\|^\beta\\
				&\leq 2 D_n (1+\delta +\|\ph_0\|_{n,\beta})^{n}\|x-y\|^\beta,
			\end{split} 
		\end{align} 
		where we use $\|d\ph(x)-d\ph(y)\|_{L_1}\leq \|\ph\|_{1,\beta}\|x-y\|^\beta$ for the case $n = 1$ 
		(which holds by definition). For the case $n \geq 2$ we use the mean value inequality.
		(If $\|x-y\| \le 1$ then $\|x-y\| \le \|x-y\|^\be$, otherwise 
		$\|d\ph(x)-d\ph(y)\|_{L_1} \le 2 \|\ph\|_{n,\be} \le 2 \|\ph\|_{n,\be} \|x-y\|$.) 
		By \Cref{help},
		\begin{align*}
			\|d^n \tau(\Ph(x)) - d^n\tau(\Ph(y))\|_{L_n}
			&\leq \|d^n \tau(\Ph(x))(A^n) - d^n\tau(\Ph(y))(A^n)\|_{L_n} \|d\Ph(x)^{-1}\|_{L_1}^n\\
			&\leq C \|d^n \tau(\Ph(x))(A^n) -  d^n\tau(\Ph(y))(A^0)\|_{L_n}\\  
			&\quad+ C\sum_{k=1}^{n}\|d^n\tau(\Ph(y))(A^k) - d^n\tau(\Ph(y))(A^{k-1})\|_{L_n}.
		\end{align*}
		We may use \eqref{unicon} to estimate the second term on the right-hand side. 
		Thus, to end the proof of Claim \ref{claim3}, and hence of the proposition, it remains to show the following.

		\begin{claim} \label{claim2}
			There exists a constant $C$ such that
		\begin{align*}
			\|d^n\tau(\Ph(x))(A^n ) 
			- d^n\tau(\Ph(y))(A^0)\|_{L_n} 
			\leq C \|x-y\|^\beta,
		\end{align*}
		 for all $\ph \in B^{n,\beta}(\ph_0,\delta)$ and all $x,y\in \mathbb{R}^d$.
		\end{claim}

		For any $\ga \in \N_{>0}^l$ and $0 \le j \leq l$ let
		\[
		A^{\gamma, j}=
		A^{\gamma, j}(x,y):= \big( \Ph^{(\ga_1)}(x), \dots,\Ph^{(\ga_{j})}(x), 
		\Ph^{(\ga_{j+1})}(y), \dots,  \Ph^{(\ga_{l})}(y)\big).
		\]
		Then, by Fa\`a di Bruno's formula \eqref{eq:FaadiBruno}, 
		\begin{align}
				\label{faa2}
				d^n(\tau\circ \Ph)(x) &- d^n(\tau\circ \Ph)(y) \notag 
				= \tau^{(n)}(\Ph(x))(A^{n}) - \tau^{(n)}(\Ph(y))(A^{0}) \notag\\
				&\quad+ \on{sym} \sum_{l=1}^{n-1} \sum_{\ga \in \Ga(l,n)}
					 c_{\ga} \big(\tau^{(l)}(\Ph(x))(A^{\ga,l})  
				- \tau^{(l)}(\Ph(y))(A^{\ga,0}) \big).
		\end{align}
		By \eqref{eq:equation}, there is a constant $C$ such that
		\begin{align} \label{eq:est1}
			\|d^n(\tau\circ \Ph)(x) - d^n(\tau\circ \Ph)(y)\|_{L_n} \leq C\|x-y\|^\beta, 
			\quad \ph \in B^{n,\beta}(\ph_0,\delta).
		\end{align}
		Moreover,
		\begin{align*}
				\MoveEqLeft
				\big\| \tau^{(l)}(\Ph(x))(A^{\gamma,l}) -  \tau^{(l)}(\Ph(y))(A^{\gamma,0})\big\|_{L_n}\\
			&\leq \big\| \tau^{(l)}(\Ph(x))(A^{\gamma,l}) - \tau^{(l)}(\Ph(y))(A^{\gamma,l})\big\|_{L_n}\\
				& \quad + 
				\sum_{k=1}^{l} \big\| \tau^{(l)}(\Ph(y))(A^{\gamma,k}) 
				- \tau^{(l)}(\Ph(y))(A^{\gamma,k-1})\big\|_{L_n}.
		\end{align*}
		For the first summand, since $l < n$, 
		\begin{align*}
			\MoveEqLeft
			\big\| \tau^{(l)}(\Ph(x))(A^{\gamma,l}) - \tau^{(l)}(\Ph(y))(A^{\gamma,l})\big\|_{L_n}
			\\
			&\le
			\big\| \tau^{(l)}(\Ph(x)) - \tau^{(l)}(\Ph(y))\big\|_{L_l}
			 (1+\|\ph\|_{n,\be})^n
			 \\
			 &\leq \|\tau\|_{n}(1+\|\ph\|_{1}) \|x-y\| (1+\|\ph\|_{n,\be})^n,
		\end{align*}
		and $\|\ta\|_{n} \le D_n$ for $\ph \in B^{n,\be}(\ph_0,\de)$, by Claim \ref{claim1}. For the other summands 
		observe that
		\begin{align*}
			\MoveEqLeft
			\tau^{(l)}(\Ph(y))(A^{\gamma,k}) - \tau^{(l)}(\Ph(y))(A^{\gamma,k-1})\\
			&=  \tau^{(l)}(\Ph(y))\big(  \dots, \Ph^{(\ga_{k-1})}(x),
			 (\Ph^{(\ga_k)}(x) -\Ph^{(\ga_k)}(y) ),\Ph^{(\ga_{k+1})}(y),\dots \big),
		\end{align*}
		whence 
		\begin{align*}
			\MoveEqLeft
			\big\|\tau^{(l)}(\Ph(y))(A^{\gamma,k}) - \tau^{(l)}(\Ph(y))(A^{\gamma,k-1})\big\|_{L_n}
			\\
			&\le  
			\begin{cases}
				\|\tau\|_{n} (1+\|\ph\|_{n})^{n-1}  \|\ph\|_{n}\|x-y\| & \text{ if } l >1,
				\\
				\|\tau\|_{1} \|\ph\|_{n,\be} \|x-y\|^\be & \text{ if } l =1.
			\end{cases}
		\end{align*}
		Altogether this means that we find a constant $K$ 
		such that for all $\ph \in B^{n,\beta}(\ph_0,\delta)$ and all $x,y \in \mathbb{R}^d$ 
		\begin{align*}
			\big\| \tau^{(l)}(\Ph(x))(A^{\gamma,l}) -  \tau^{(l)}(\Ph(y))(A^{\gamma,0})\big\|_{L_n}\leq K\|x-y\|^\be;
		\end{align*}
		indeed, if $\|x-y\| \le 1$ then $\|x-y\| \le \|x-y\|^\be$, otherwise the estimate follows from the triangle inequality 
		and Claim \ref{claim1}.  
		Together with \eqref{faa2} and \eqref{eq:est1} this implies Claim \ref{claim2}.
	\end{proof}
	
For later use we prove the following.

\begin{proposition}
		\label{prop:invHoelder}
		Let $n \in \N_{\ge 1}$ and $0< \al < \be \le 1$.
		Then, for all $\ph_0 \in (\mathcal{D}^{n,\beta}(\R^d)-\Id)$ there exists $\de>0$ such that 
		for all $\ph_1,\ph_2 \in B^{n,\be}(\ph_0,\de) \subseteq (\mathcal{D}^{n,\beta}(\R^d)-\Id)$, 
		\begin{equation}
			\|\inv_c (\ph_1) - \inv_c (\ph_2)\|_{n,\al} \le M \|\ph_1 - \ph_2\|_{n,\al}^{\be-\al}, 	
		\end{equation} 
		where $M = M(n,\ph_0,\de)$.
		In particular, 
		\begin{align*}
			\inv_c : (\mathcal{D}^{n,\beta}(\R^d)-\Id) \rightarrow (\mathcal{D}^{n,\al}(\R^d)-\Id),
			\quad \ph \mapsto (\Id+\ph)^{-1} -\Id
		\end{align*}
		is continuous.
	\end{proposition}

\begin{proof}
	Choose $\de >0$ such that $B^{n,\be}(\ph_0,\de) \subseteq (\mathcal{D}^{n,\beta}(\R^d)-\Id)$. 
	Let $\ph_1,\ph_2 \in B^{n,\be}(\ph_0,\de)$. 
	We write $\tau_i = \inv_c(\ph_i)$ and $\Ph_i= \Id+\ph_i$, for $i =0,1,2$.
			Then
			\begin{align*}
				\tau_1 -\tau_2 
				&= \Ph_1^{-1}\circ \Phi_2 \circ \Phi_2^{-1} -\Phi_1^{-1}\circ \Phi_1 \circ \Phi_2^{-1}\\
				&= \ph_2 \circ (\Id + \tau_2) - \ph_1\circ(\Id + \tau_2) 
				\\
				&\quad + \tau_1\circ(\Id+\ph_2)\circ(\Id+\tau_2) 
				- \tau_1\circ(\Id+\ph_1)\circ(\Id+\tau_2).
			\end{align*}
			By \Cref{comp1}, 
			\begin{align*}
				\|\ph_2 \circ (\Id + \tau_2) - \ph_1\circ(\Id + \tau_2)\|_{n,\alpha} 
				\leq M(n) \|\ph_1 - \ph_2\|_{n,\alpha} (1+\|\tau_2\|_{n,\alpha})^{n+1},
			\end{align*}
			and by \Cref{comp1}, \Cref{6.2} (and \Cref{inclusion}),
			\begin{align*}
				\MoveEqLeft
				\|\tau_1\circ(\Id+\ph_2)\circ(\Id+\tau_2) - \tau_1\circ(\Id+\ph_1)\circ(\Id+\tau_2)\|_{n,\alpha}\\
				\leq & ~ M(n)  \|\tau_1\circ(\Id+\ph_2) - \tau_1\circ(\Id+\ph_1)
				\|_{n,\alpha} (1+\|\tau_2\|_{n,\alpha})^{n+1}\\
				\leq & ~ M(n,\|\ph_0\|_{n,\al} ,\de) \|\tau_1\|_{n,\be} 
				\|\ph_1 - \ph_2\|^{\be -\al}_{n,\alpha} (1+\|\tau_2\|_{n,\alpha})^{n+1}.
			\end{align*}
			Since $\|\tau_i\|_{n,\be}$ is uniformly bounded for $\ph_i \in B^{n,\be}(\ph_0,\delta)$ 
			if $\de >0$ is chosen sufficiently small, 
			by \Cref{invcl}, this implies the assertion.
\end{proof}

	\subsection{Intermediate H\"older spaces}

	As we have already seen in \Cref{disc}, 
	the group $\mathcal{D}^{n,  \beta}(\R^d)$ is not topological 
	(with respect to the topology given as a Banach manifold modelled on $C^{n,\beta}_0(\R^d,\R^d)$). 
	Nevertheless we know that the left translations become continuous if the outer mapping is only 
	slightly more regular than the space it acts on. This observation motivates the following definitions.

	Let $E$, $F$ be Banach spaces, $U \subseteq E$ open, and $n \in \N$.
	For $\be \in (0,1]$ define 
	\[
	C^{n,\beta-}_b(U,F):= \bigcap_{\alpha \in (0,\be)} C^{n,\alpha}_b(U,F), 
	\] 	
	and for 
	$\be \in [0,1)$,
	\[
	C^{n,\beta+}_b(U,F):= \bigcup_{\alpha \in (\be,1)} C^{n,\alpha}_b(U,F). 
	\]
	If $\be \in (0,1)$ we have the strict inclusions
	\begin{equation*}
		C^{n,\be+}_b(U,F) \subsetneq C^{n,\be}_b(U,F)  \subsetneq C^{n,\be-}_b(U,F).
	\end{equation*}
	We endow $C^{n,\be-}_b(U,F)$ and $C^{n,\be+}_b(U,F)$ with their natural projective and inductive 
	locally convex limit topologies, 
	respectively. 

	Then $C^{n,\beta-}_b(U,F)$ is a Fr\'echet space with a generating system 
	of seminorms $\mathcal{P}=\{ \|\cdot\|_{n,\alpha}: \alpha \in (0,\be) \}$, 
	or a countable subfamily thereof, like $\{ \|\cdot\|_{n,\beta - 1/k}: k \geq k_0 \}$.
	The balls $B_{\alpha}^{n, \beta-}(f_0,\varepsilon):= 
	\{ f \in C^{n,\beta-}_b (U,F): \|f  -f_0\|_{n,\alpha} 
		< \varepsilon \}$ satisfy 
		\[
		B_{\alpha_2}^{n,\beta-}(f_0,\varepsilon) \subseteq B_{\alpha_1}^{n, \beta-}(f_0,2\varepsilon)
		\quad \text{ if } \al_1 < \al_2, 
		\]
		by \Cref{inclusion}.
		Thus $\{ B^{n,\beta-}_{\alpha}(f_0,\varepsilon): \alpha < \beta,\, \varepsilon > 0 \}$
		forms a neighborhood base
		of $f_0 \in C^{n,\beta-}_b(U,F)$.

	In analogy we define $C^{n,\be\pm}_0$ and $C^{n,\be\pm}$.		

	\begin{lemma} \label{lem:cpreg}
		$C^{n,\beta+}_b(U,F)$ and $C^{n,\beta+}_0(E,F)$ are compactly regular (LB)-spaces.		
	\end{lemma}	

	\begin{proof}
		  It suffices, by \cite[Satz 1]{Neus78}, to verify condition (M) of \cite{Retakh70}:
  There exists a sequence of increasing 0-neighborhoods $B_p\subseteq C^{n,\be+1/p}_b(U,F)$ such that for 
  each $p$ there exists an $m\geq p$ for which the topologies of $C^{n,\be +1/k}_b(U,F)$ and
  of $C^{n,\be +1/m}_b(U,F)$ coincide on $B_p$ for all $k\geq m$.

  For $\al \le \al'$ we have $\|f\|_{n,\al} \le 2\|f\|_{n,\al'}$, by \Cref{inclusion}. 
  It suffices to show that for 
  $\be < \al_2 < \al_1 < \al$,
  $\ep>0$, and 
  $f\in B^{n,\al}(0,1)$ there exists $\de>0$ such that 
  $B^{n,\al_2}(f,\de)\cap B^{n,\al}(0,1)\subseteq B^{n,\al_1}(f,\ep)$.
  
  Let $g \in B^{n,\al_2}(f,\de)\cap B^{n,\al}(0,1)$. Then $\|g-f\|_{n,\al_2} < \de$ and $\|g\|_{n,\al} < 1$. 
  By \Cref{exp},
  \begin{align*}
  	\|g-f\|_{n,\al_1} 
  	&\le  \|g-f\|_{n,\al_2}^{\frac{\al-\al_1}{\al-\al_2}} \|g-f\|_{n,\al}^{\frac{\al_1-\al_2}{\al-\al_2}} 
  	<  \de^{\frac{\al-\al_1}{\al-\al_2}}  2^{\frac{\al_1-\al_2}{\al-\al_2}}.
  \end{align*}
  So it is clear that we may find $\de$ as required.	
 \end{proof}

 Consequently, $C^{n,\beta+}_b(U,F)$ and $C^{n,\beta+}_0(U,F)$ are complete (thus convenient), webbed, and ultra-bornological.

\subsection{\texorpdfstring{$C^{0,\om}$}{C0omega}-mappings between convenient vector spaces}		
\label{sec:C0omconvenient}

Let $\om : [0,\infty) \to [0,\infty)$ be a subadditive increasing modulus of continuity ($\lim_{t \to 0}\om(t) = \om(0) = 0$). 
By a $C^{0,\om}$-curve $c$ we mean a function defined on the real line with values in a convenient vector space $F$ such that
for each bounded interval $I \subseteq \R$,
\[
\Big\{\frac{c(t) - c(s)}{\om(|t-s|)} : t,s \in I,\, t \ne s \Big\}
\]
is bounded in $F$.
We say that a mapping between convenient vector spaces is $C^{0,\om}$, if it maps $C^\infty$-curves to $C^{0,\om}$-curves. 
The $c^\infty$-topology coincides with the final topology of all $C^{0,\om}$-curves 
(which follows from the proof of \cite[2.13]{KM97}), and so a $C^{0,\om}$-mapping is continuous with respect to 
the $c^\infty$-topology.  
The following lemma shows that between Banach spaces the notion of $C^{0,\om}$-mapping coincides with the usual definition. 

\begin{lemma} \label{lem:C0om}
	Let $E$, $F$ be Banach spaces, $U \subseteq E$ open. 
	A mapping $f : U \to F$ is $C^{0,\om}$ if and only if $f(x) - f(y)/\om(\|x-y\|)$ is locally bounded. 
\end{lemma}

\begin{proof}
	Suppose that there is $z \in U$ and $x_n \ne y_n \in U$ such that $\|x_n - z \| \le 4^{-n}$, 
	$\|y_n - z \| \le 4^{-n}$, and $\|f(x_n) - f(y_n)\| \ge n 2^n \om(\|x_n - y_n\|)$.
	By \cite[12.2]{KM97}, there is a $C^\infty$-curve $c$ and a convergent sequence of real numbers $t_n$ such that 
	$c(t+t_n) = x_n + t \frac{(y_n - x_n)}{2^{n} \|x_n - y_n\|}$ for all $0 \le t \le s_n := 2^n \|x_n - y_n\|$. 
	Then, by subadditivity of $\om$, 
	\[
		\frac{\|(f\o c)(t_n + s_n) - (f\o c)(t_n)\|}{\om(s_n)} 
		= \frac{\|f(x_n) - f(y_n)\|}{\om(2^n \|x_n - y_n\|)}
		\ge n.
	\]
	The converse implication follows from subadditivity and monotonicity of $\om$, 
	since $C^\infty$-curves are locally Lipschitz.
\end{proof}

This lemma can be found in \cite{Faure91}, \cite[12.7]{KM97}, or \cite[Lemma]{KMR} in the H\"older (or Lipschitz) case
$\om(t) = t^\ga$. 

\begin{definition}
	We say that $\om$ is a \emph{slowly vanishing} modulus of continuity if $\om$ 
	is increasing, subadditive, and satisfies  
	\[
		\liminf_{t\downarrow 0}  \frac{\om(t)}{t^\ga} > 0  \quad \text{ for all } \ga>0.
	\]
\end{definition}
For instance, $\om$ defined by 
 $\om(t) := -(\log t)^{-1}$, if $0 < t < e^{-2}$, $\om(t) := 1/2$, if $t \ge e^{-2}$, and $\om(0):= 0$,
 is a slowly vanishing modulus of continuity.

	\subsection{The \texorpdfstring{$C^{0,\om}$}{C0om} Lie groups \texorpdfstring{$\mathcal{D}^{n,\beta-}(\R^d)$}{Dnbeta-} 
	and \texorpdfstring{$\mathcal{D}^{n,\beta+}(\R^d)$}{Dnbeta+}}

	Let $n \in \N_{\ge 1}$.
	We define
	\begin{gather*}
		\cD^{n,\beta\pm}(\mathbb{R}^d):= \big\{ \Phi \in \Id + C^{n,\beta\pm}_0(\mathbb{R}^d, \mathbb{R}^d)
		 : \det d \Phi(x) > 0 ~ \forall x \in \mathbb{R}^d \big\},
	\end{gather*}
	where $\be \in (0,1]$ if $\pm = -$ and $\be \in [0,1)$ if $\pm = +$. 
	Then $\mathcal{D}^{n,\beta\pm}(\R^d) - \Id$ is an open subset of $C^{n,\beta\pm}_0(\R^d,\R^d)$.
	We take this interpretation as defining property for the topology, i.e., 
	$V \subseteq \mathcal{D}^{n,\beta\pm}(\R^d)$ is open if and only if $(V - \Id)$ is open 
	in $(\mathcal{D}^{n,\beta\pm}(\R^d)-\Id)\subseteq C^{n,\beta\pm}_0(\R^d,\R^d)$.
	
	Clearly, $\cD^{n,\be\pm}(\R^d)$ forms a group, by \Cref{thm:group}.
	We will now prove that $\cD^{n,\be\pm}(\R^d)$ are  
	$C^{0,\om}$ Lie groups for any slowly vanishing modulus of continuity $\om$.

\begin{theorem} \label{thm:C0omegaLie}
	Let $n \in \N_{\ge 1}$. 
	Let $\om$ be a slowly vanishing modulus of continuity. 
	Then 
	$\cD^{n,\be-}(\R^d)$, for $\be \in (0,1]$, and $\cD^{n,\be+}(\R^d)$, for $\be \in [0,1)$,  
	are $C^{0,\om}$ Lie groups. 
	In particular, $\cD^{n,\be-}(\R^d)$, for $\be \in (0,1]$, is a topological group (with respect to its 
	natural Fr\'echet topology).
\end{theorem}

\begin{proof}
	Let us first consider $\cD^{n,\be-}(\R^d)$, for $\be \in (0,1]$.
	Let $g,f \in C^\infty(\R,C^{n,\be-}_0(\R^d,\R^d))$ and let $I \subseteq \R$ be a compact interval. 
	Then the sets $g(I)$, $f(I)$ are bounded in $C^{n,\be-}_0(\R^d,\R^d)$ and thus in every $C^{n,\al}_0(\R^d,\R^d)$ for 
	$\al<\be$. If $\al < \tilde \al < \be$, then, by \eqref{eq:jcont},
	\begin{align*}
	 	\MoveEqLeft
	 	\|g(t) \o (\Id + f(t))-g(s)\o  (\Id + f(s))\|_{n,\al} 
	 		\\
			&\le M\big(\|g(t) - g(s)\|_{n,\al} + \|f(t)-f(s)\|_{n,\al}\big)^{\tilde \al - \al}
			\le \tilde M |t - s|^{\tilde \al - \al},
	 \end{align*}
	for $t,s \in I$. There is $\ep >0$ and $C=C(\al,\tilde \al)$ such that $|t - s|^{\tilde \al - \al} \le C \om(|t-s|)$ 
	if $|t-s| \le \ep$. Since $\om$ is increasing, we may conclude that, for $t \mapsto h(t) :=g(t) \o (\Id + f(t))$, 
	\begin{equation} \label{eq:set}
		\Big\{\frac{h(t) -h(s)}{\om(|t-s|)} : s \ne t \in I\Big\}
	\end{equation}
	is bounded in $C^{n,\al}_0(\R^d,\R^d)$. 
	So the composition is $C^{0,\om}$ on $\cD^{n,\be-}(\R^d)$.

	Let us turn to the inversion in $\cD^{n,\be-}(\R^d)$. 
	Let $f \in C^\infty(\R,C^{n,\be-}_0(\R^d,\R^d))$. 
	Fix $\al < \tilde \al < \be$ and $t_0 \in \R$. 
	Let $\de>0$ be such that $B^{n,\tilde \al}(f(t_0),\de) \subseteq (\cD^{n,\tilde \al}(\R^d) - \Id)$. 
	There is a neighborhood $I$ of $t_0$ such that $f(I) \subseteq B^{n,\tilde \al}(f(t_0),\de)$. 
	By \Cref{prop:invHoelder} (after possibly shrinking $\de$), for all $t,s \in I$, 
	\[
		\|\inv_c (f(t)) - \inv_c (f(s))\|_{n,\al} \le M \|f(t) - f(s)\|_{n,\al}^{\tilde \al- \al},
	\]
	where $M = M(n,f(t_0),\de)$. 
	Finishing the arguments in the same way as for the composition, 
	we conclude that the inversion is $C^{0,\om}$ on $\cD^{n,\be-}(\R^d)$. 

	This implies that $\cD^{n,\be-}(\R^d)$ is a topological group, since the underlying Fr\'echet topology and 
	the $c^\infty$-topology coincide. Of course, it also follows directly from \Cref{jcont} and \Cref{prop:invHoelder}.

	Now let us consider $\cD^{n,\be+}(\R^d)$, for $\be \in [0,1)$.
	Let $g,f \in C^\infty(\R,C^{n,\be+}_0(\R^d,\R^d))$.
	For any compact interval $I \subseteq \R$, the images $g(I)$, $f(I)$ are bounded in $C^{n,\be+}_0(\R^d,\R^d)$. 
	Since $C^{n,\be+}_0(\R^d,\R^d)$ is a compactly regular (LB)-space, there is some $\al_0>\be$ such that 
	$g(I)$, $f(I)$ are bounded in $C^{n,\al_0}_0(\R^d,\R^d)$, and thus also in every $C^{n,\al}_0(\R^d,\R^d)$, for 
	$\al \in (\be,\al_0]$. 
	Let $\al,\tilde \al \in (\be,\al_0]$ with $\al < \tilde \al$. Then the arguments above  
	show that the set \eqref{eq:set} is bounded in $C^{n,\al}_0(\R^d,\R^d)$, and thus in $C^{n,\be+}_0(\R^d,\R^d)$.
	So the composition is $C^{0,\om}$ on $\cD^{n,\be+}(\R^d)$. Similarly for the inversion.   
\end{proof}

\begin{remark}
	We do not know whether $\cD^{n,\be+}(\R^d)$ is a topological group with respect to its natural 
	inductive locally convex topology, since the $c^\infty$-topology is finer in this case.	
\end{remark}

Groups with continuous left translations and smooth right translations were dubbed half-Lie groups in 
	\cite{KrieglMichorRainer14a}. The chart representations of the right translations in $\cD^{n,\be\pm}(\R^d)$ 
	are affine and bounded, by \Cref{comp1}, 
	and thus smooth.
	Hence, $\cD^{n,\be-}(\R^d)$ is a half-Lie group.

The next result shows that the $C^{0,\om}$-regularity of the group operations in $D^{n,\be\pm}(\R^d)$ 
is optimal.

\begin{proposition} \label{cor:optimal}
	Let $n \in \N_{\ge 1}$. 
	\begin{enumerate}
		\item For all $\be \in (0,1]$,  
	$\cD^{n,\be-}(\R^d)$ is a half-Lie group. 
	There are left translations in $\cD^{n,\be-}(\R^d)$ which are not locally H\"older continuous of any order $\ga>0$.
		\item Let $\be \in [0,1)$. 
	For any $\ga>0$, there are  
	left translations in $\cD^{n,\be+}(\R^d)$ which are not locally H\"older continuous of order $\ga$.
	\end{enumerate}	
\end{proposition}

\begin{proof}
	(1)
	Let $\chi \in C^\infty_c(\R)$ be $1$ on $[-1,1]$ and satisfy $\ch'(x) >-1$ for all $x \in \R$, 
	and set $\ps(x):= x^n |x|^\beta \ch(x) 
	\in C^{n,\beta}_0(\mathbb{R},\mathbb{R}) \subseteq C^{n,\beta-}_0(\mathbb{R},\mathbb{R})$. 
	We will show that $\th(t):= \ps\o (\Id+t\ch)$, for small $t \in \R$, is  
	not locally H\"older continuous of order $\ga$
	into $C^{n,\al}_0(\R,\R)$ 
	for any $\al>\be-\ga$.
	This implies the assertion, since $\Id + r \ps \in \cD^{n,\be-}(\R)$ if $r>0$ is small enough.   	
	We must show that, for any small interval $I \ni 0$, the set
	\[
		\Big\{\frac{\th(t)^{(n)}(x) - \th(t)^{(n)}(y) - \th(s)^{(n)}(x) + \th(s)^{(n)}(y)}{|x-y|^\al|s-t|^\ga} :
		x \ne y \in \R, s \ne t \in I
		\Big\}
	\]
	is unbounded. If $|x| < 1$, then for small $t$ (cf.\ \eqref{eq:constant}),
	\begin{equation*}
	 		\th(t)^{(n)}(x) = \ps^{(n)}(x+t) =  C_{n,\be}\, |x+t|^\beta.
	\end{equation*}
	For $t=x=0$ and $|y| \le 1$ the expression reads (up to a constant factor)
	\[
		\frac{ - |y|^\be - |s|^\be + |y+s|^\be}{|y|^\al|s|^\ga} 
	\]
	and upon setting $y = -s$, we get $-2 |s|^{\be -\al-\ga}$ which is unbounded near $s=0$.

	(2) Let $\ga >0$ be given.
	For $\al > \be$ let $\ps_\al(x) := x^n |x|^\al \ch(x)$. 
	Then, as seen above, $\th_\al(t):= \ps_\al \o (\Id+t\ch)$, for small $t \in \R$, is  
	not locally H\"older continuous of order $\ga$
	into $C^{n,\al_1}_0(\R,\R)$ 
	for any $\al_1 \in (\be,\al)$ with  $\ga>\al-\al_1$.
	It follows that $(\ps_\al)_\star$ is not locally H\"older continuous of order $\ga$, provided that 
	$\al -\be < \ga$. 
	Indeed, if 
	\[
		\Big\{\frac{\th_\al(t) - \th_\al(s)}{|s-t|^\ga} :
		s \ne t \in I
		\Big\}	
	\] 
	were bounded in $C^{n,\be+}_0(\R,\R)$, then it would be so in some step $C^{n,\al_1}_0(\R,\R)$, 
	by \Cref{lem:cpreg}. 
\end{proof}

The results of this section are summarized in \Cref{table:groups}.

{\small
\begin{table}[H]
\centering
\begin{tabular}{@{}lccccc@{}}
  & group &  $C^{0,\om}$-Lie group & topological group & half-Lie group & Lie group \\ \midrule
 $\cD^{n,\be}(\R^d)$ & yes & no & no & no & no \\ \midrule
 $\cD^{n,\be-}(\R^d)$ & yes & yes & yes & yes & no \\ \midrule
 $\cD^{n,\be+}(\R^d)$ & yes & yes & ? & ? & no \\ 
\end{tabular}
\caption{\label{table:groups}\small{Here $n \in \N_{\ge 1}$ and $\om$ is any slowly vanishing 
modulus of continuity. In the first two rows 
$\be \in (0,1]$, in the third row $\be \in [0,1)$.}}
\end{table}
}

\section{H\"older spaces are ODE closed} \label{Trouve}

	\subsection{Flows of time-dependent H\"older vector fields}
	
Let $n \in \N_{\ge1}$ and $\be \in (0,1]$. 	
By a \emph{strong time-dependent $C^{n,\beta}_0$-vector field} 
we mean a Bochner integrable function 
$u : [0,1] \to C^{n,\beta}_0(\R^d,\R^d)$. We will write $I := [0,1]$ and  
\[
	\|u\|_{L^1(I,C^{n,\beta})} := \int_0^1 \|u(t)\|_{n,\be} \,dt.	
\]
The space $L^1(I,C^{n,\beta}_0(\R^d,\R^d))$ of (equivalence classes with respect to a.e.\ coincidence of) Bochner integrable function 
$u : I \to C^{n,\beta}_0(\R^d,\R^d)$ equipped with this norm is a Banach space.

	Let $\al \le \be$.
	We say that a continuous mapping $\Phi:I \rightarrow \mathcal{D}^{n,\alpha}(\R^d)$ is a 
	\emph{strong $\mathcal{D}^{n,\alpha}$-flow} 
	of $u$ if for all $t \in I$ we have
	\begin{equation}
		\label{eq:Banachflow}
		\Phi(t) = \Id + \int_0^t u(s) \circ \Phi(s)\,ds
	\end{equation}
	in $\mathcal{D}^{n,\alpha}(\R^d)$, where the integral is the Bochner integral.

	Since evaluation $\ev_x$ at $x \in \R^d$ is continuous and linear on $C^{n,\al}_0(\R^d,\R^d)$ and it thus commutes with the 
	Bochner integral, 
	\eqref{eq:Banachflow} entails
	\begin{equation}
		\label{eq:pointwiseflow}
		\Phi^\wedge(t,x) = x + \int_0^t u^\wedge(s,\Phi^\wedge(s,x))\,ds, \quad x \in \R^d.
	\end{equation}
	We say that $\Ph^\wedge : I \times \R^d \to \R^d$ is the \emph{pointwise flow} of $u^\wedge$ if it satisfies 
	\eqref{eq:pointwiseflow}. 
	So, if $u$ has a strong $\cD^{n,\alpha}$-flow $\Ph$, then $u^\wedge$ has a pointwise flow which is continuous in $t$ and 
	differs from the identity by a $C^{n,\al}_0$-mapping in $x$. 
	Conversely, the existence of a pointwise flow with this properties will entail the existence of a 
	strong $\cD^{n,\alpha}$-flow only if the Bochner integral in \eqref{eq:Banachflow} exists.  
	Since the H\"older spaces are non-separable, strong measurability of $t \mapsto u(t) \o \Ph(t)$ may fail 
	and the integral in \eqref{eq:Banachflow} may not exist,
	if the left translation $u(t)_\star$ is not continuous.
	This is exactly what happens if $\be=\al$.

	Luckily we can work with pointwise estimates which enable us to prove that time-dependent $C^{n,\be}_0$-vector fields
	have unique pointwise flows $\Ph$ such that $\Ph^\vee \in C(I,\cD^{n,\be}(\R^d))$ (no loss of regularity!).  
	The proof actually works for a wider class of vector fields, 
	so-called \emph{pointwise time-dependent $C^{n,\be}_0$-vector fields}, 
	which shall be introduced in the next subsection.

	We shall see in \Cref{ssec:strongflow} that the unique pointwise flow $\Ph^\vee \in C(I,\cD^{n,\be}(\R^d))$ of a 
	strong time-dependent $C^{n,\be}_0$-vector field $u$ lifts to a strong $\cD^{n,\al}$-flow, for each $\al < \be$.  

\subsection{Trouv\'e group and ODE closedness} \label{sec:ODEhull}

Let $I = [0,1]$ and let $E$ be a Banach space of mappings $\R^d \to \R^d$ which is continuously embedded in $C^1_0(\R^d,\R^d)$.

\begin{definition}
	We say that a mapping $u : I \times \R^d \to \R^d$ is a \emph{pointwise time-dependent $E$-vector field} if 
	the following conditions are satisfied.
	\begin{itemize}
		\item $u(t, \cdot) \in E$ for every $t \in I$.
		\item $u(\cdot,x)$ is measurable for every $x\in \R^d$.
		\item $I \ni t \to \|u(t,\cdot)\|_{E}$ is (Lebesgue) integrable. 
	\end{itemize}
	Let us denote the set of all pointwise time-dependent $E$-vector fields by $\fX_E(I,\R^d)$. 
	We remark that instead of the third condition we could also require that $\|u^\vee\|_E$ is dominated a.e.\ by some 
	non-negative function 
	$m \in L^1(I)$.
\end{definition}

Clearly, $u \in L^1(I,E)$ implies $u^\wedge \in \fX_E(I,\R^d)$; the converse is in general not true, in particular, if $E$ 
is non-separable and strong measurability and measurability are not the same; see \Cref{example:strongmeasurability} below. 
We will continue to write 
\[
	\|u^\vee\|_{L^1(I,E)} = \int_0^1 \|u^\vee(t)\|_{E} \,dt, \quad \text{ for  } u \in \fX_E(I,\R^d),
\]
even though $\|u^\vee\|_{L^1(I,E)}$ might be finite while $u^\vee : I \to E$ is not Bochner integrable; 
this will lead to no confusion. 

\begin{example} \label{example:strongmeasurability}
	Let $\chi \in C^\infty_c(\R)$ be $1$ on $[-1,1]$, and let $\ps(x):= x^n |x|^\beta \ch(x)$, then $\ps$  lies in $C^{n,\be}_0(\R,\R)$ (cf. \Cref{disc}). 
	Let $u : I \times \R \to \R$ be defined by $u(t,x) = \ps(x-t)$; $u$ is clearly a pointwise time-dependent $C^{n,\be}_0$-vector field.
	But $u^\vee \not\in L^1(I,C^{n,\be}_0(\R))$:
	indeed, for fixed $t,s \in I$, $t \ne s$, (cf.\ \eqref{eq:constant})
	\begin{align*}
		[u^\vee(t) - u^\vee(s)]_{n,\be} 
		&= \sup_{x \ne y} \frac{|\ps^{(n)}(x-t) - \ps^{(n)}(y-t) - \ps^{(n)}(x-s) + \ps^{(n)}(y-s)|}{|x-y|^\be} 
		\\
		&\ge C_{n,\be}
		\sup_{x,y \in I,\, x \ne y} \frac{\big||x-t|^\be - |y-t|^\be - |x-s|^\be + |y-s|^\be \big|}{|x-y|^\be}
		\\
		&\ge 2 C_{n,\be} > 0 \quad \text{ (choose $x=t$, $y=s$).} 
	\end{align*} 
	It follows that the image $u^\vee(I)$ is not essentially separable in $C^{n,\be}_0(\R)$, and so $u^\vee$ is not 
	strongly measurable, by the Pettis measurability theorem (cf.\ \cite[p.~42]{DiestelUhl77}). 
\end{example}

It is well-known that pointwise time-dependent $C^n_0$-vector fields $u$ 
have unique pointwise flows $\Ph = \Ph_u : I \times \R^d \to \R^d$ 
such that $\Ph^\vee : I  \to \Id + C^{n}_0(\R^d,\R^d)$ is continuous, and $\Ph^\vee(t)$ is a $C^n$-diffeomorphism at any time 
$t$; see e.g.\ \cite[8.7, 8.8, 8.9]{Younes10} and the arguments in the proof of \Cref{thm:Trouve1} below. 

\begin{definition} \label{def:Trouve}
	Let $E$ be a Banach space of mappings $\R^d \to \R^d$ which is continuously embedded in $C^1_0(\R^d,\R^d)$.
	Then 
	\[
		\cG_E := \big\{ \Ph_u^\vee(1) : u \in \fX_E(I,\R^d)  \big\}
	\]
	is a group with respect to composition; cf.\ \cite[8.14]{Younes10}. We call $\cG_E$ the \emph{Trouv\'e group} of $E$.
\end{definition}

\begin{definition} \label{def:ODEclosed}
	We say that $E$ is \emph{ODE closed} if $\cG_E \subseteq \Id + E$.
\end{definition}

\begin{remark}
	It is clear that, more generally, we could take (mutatis mutandis) 
	any locally convex space $E$ of mappings $\R^d \to \R^d$ 
	which is continuously embedded in $C^1_0(\R^d,\R^d)$ in the above definitions. 

	Furthermore, this leads to the notion of \emph{ODE hull} of $E$, i.e., the intersection of all 
	locally convex spaces $F$ of mappings $\R^d \to \R^d$ 
	which are continuously embedded in $C^1_0(\R^d,\R^d)$ and continuously contain $E$, 
	endowed with the natural projective topology.
	The ODE hull is well-defined, because $C^1_0$ is ODE-closed, and it is evidently ODE closed. 
\end{remark}

\subsection{The Trouv\'e group of \texorpdfstring{$C^{n,\be}_0(\R^d,\R^d)$}{Cnalpha0}}

Let $n \in \N_{\ge1}$ and $\be \in (0,1]$.
In this section we show that the Trouv\'e group of $C^{n,\be}_0(\R^d,\R^d)$,
\[
	 \cG_{n,\be}(\R^d) := \cG_{C^{n,\be}_0(\R^d,\R^d)},
\]
coincides with the connected component of the identity in $\cD^{n,\be}(\R^d)$. In particular, 
$C^{n,\be}_0$ is ODE closed. Let us use the short notation
\[
	\fX_{n,\be}(I,\R^d) := \fX_{C^{n,\be}_0}(I,\R^d).
\] 
We want to stress that this is an example of an ODE closed space on which left translations $g_\star$ 
are \emph{not} continuous.

\begin{theorem} \label{thm:Trouve1}
	Let $n \in \N_{\ge 1}$ and $\be \in (0,1]$.
	Let $u \in \fX_{n,\be}(I,\R^d)$. 
	Then $u$ has a unique pointwise flow $\Ph : I \times \R^d \to \R^d$ such that 
	$\Ph^\vee : I \to \cD^{n,\be}(\R^d)$ 
	is continuous. In particular, $C^{n,\be}_0$ is ODE closed.
\end{theorem}

\begin{proof}
	In fact, cf. \cite[8.7, 8.8, 8.9]{Younes10}, the pointwise flow exists, $x \mapsto \Ph(t,x)$ is $C^n_b$, and for all $t \in I$,
	\begin{equation} \label{eq:Younes}
		\|\Ph(t,\cdot)-\Id\|_n \leq C_1 e^{C_2\|u\|_{L^1(I, C^n_b)}}.
	\end{equation}
	 Set $W^n_x(t):= d_x^n \Ph(t,x)$ and 
	 $V^n_{x,y}(t):= (W^n_x(t)-W^n_y(t))/\|x-y\|^\be$.
	Then $W_x^n(t)$ satisfies
		\begin{align*}
			\partial_t W^n_x(t) = d_x^n\big(u^\vee(t)(\Ph(t,x))\big), \quad  W^n_x(0) &=\left\{\begin{array}{cl}
			\mathbb 1, & \text{for } n = 1\\
			0, & \text{for } n \geq 2
			\end{array}\right..
		\end{align*}
	Upon setting 
	\[
	A^{\ga,j}(t)= A^{\ga,j}(x,y)(t) :=  \big(W_x^{\ga_1}(t), \dots, W_x^{\ga_j}(t), 
	W_y^{\ga_{j+1}}(t),  \dots, W_y^{\ga_l}(t) \big).
	\]	
	and using Fa\`a di Bruno's formula \eqref{eq:FaadiBruno}, this ODE takes the form
	\begin{align*}
				\partial_t W^n_x(t)
				= \on{sym} \sum_{l=1}^{n} \sum_{\ga \in \Ga(l,n)}
				 c_{\ga} u^\vee(t)^{(l)}(\Ph(t,x))(A^{\ga,l}(t)),
	\end{align*}
	and analogously,
	\begin{align*}
				\partial_t W^n_y(t)
				= \on{sym} \sum_{l=1}^{n} \sum_{\ga \in \Ga(l,n)}
				 c_{\ga} u^\vee(t)^{(l)}(\Ph(t,y))(A^{\ga,0}(t)).
	\end{align*} 
	It follows that $V^n_{x,y}(t)$ satisfies $V^n_{x,y}(0)=0$ and
	\begin{align} \label{eq:Vode}
		\p_t V^n_{x,y}(t) &= A_x(t)\cdot V_{x,y}^n(t) + b^n_{x,y}(t) \notag
		\\&~  + \on{sym} \sum_{l=2}^{n} \sum_{\ga \in \Ga(l,n)}  c_{\gamma} 
						\frac{u^\vee(t)^{(l)}(\Ph(t,x))\cdot A^{\ga,l}(t)
						- u^\vee(t)^{(l)}(\Ph(t,y))\cdot A^{\gamma,0}(t)}{\|x-y\|^{\be}},
	\end{align}
	where $A_x(t) = du^\vee(t)(\Ph(t,x))$ and
	$
		b^n_{x,y}(t) := \frac{A_x(t) - A_y(t)}{\|x-y\|^\be} \cdot W_y^n(t).
	$
	It can be easily seen, using \eqref{eq:Younes}, that
	\begin{equation}
		\|A_x(t)\|_{L_1}\leq \|u^\vee(t)\|_{1},\quad	\|b_{x,y}^n(t)\|_{L_1} \le \|u^\vee(t)\|_{1,\be} [\Ph^\vee(t)]_1^{\be} [\Ph^\vee(t)]_n \le C_3\|u^\vee(t)\|_{1,\be}.
	\end{equation}
 	Similar arguments (see, e.g., the proof of \Cref{invcl}) show that all remaining terms 
 	in the sum can be estimated by $\|u(t)\|_{n,\be}$ times a constant
 	uniformly in $x,y$. An application of Gronwall's inequality implies that $[\Ph^\vee(t)]_{n,\be}$ is bounded in $t$, 
 	showing that $\Ph^\vee(t) \in \Id+ C^{n,\be}_b(\R^d,\R^d)$ for all $t$. Finally, we may conclude, 
 	integrating \eqref{eq:Vode} and 
 	using similar estimates, that
 	\begin{align*}
		[\Ph^\vee(t) - \Ph^\vee(t_0)]_{n,\be} = \sup_{x\ne y} \|V_{x,y}^n(t) - V_{x,y}^n(t_0) \|_{L_n}  
		&\le C_4 \int_{t_0}^t \|u^\vee(s)\|_{n,\be}  \,ds   
	\end{align*}
	which tends to $0$ 
	as $t \to t_0$. In an analogous way one sees that $\|\Ph^\vee(t) - \Ph^\vee(t_0)\|_{n} \to 0$ as $t \to t_0$. 
	This shows continuity in time.

	It remains to prove that $\Ph^\vee(I) \subseteq \cD^{n,\be}(\R^d)$. 
		For fixed $x \in \R^d$, the mapping $I \ni t \mapsto \det d \Ph(t,x)$ is continuous 
		with image in $\R \setminus \{0\}$ (since $\Ph^\vee(t)$ is a $C^1$-diffeomorphism of $\R^d$ for each $t$).
		Since $\Ph^\vee(0) = \Id$, we may conclude that $\det d \Ph(t,x) > 0$ for all $x \in \R^d$ and all $t \in I$.

	Finally, let us check that $\ph^\vee(t) = \Ph^\vee(t) - \Id \in C^{n,\be}_0(\R^d,\R^d)$ for all $t \in I$. 
	Suppose that $\|\ph(t,x)\| \not \to 0$ as $\|x \| \to \infty$. Then there is $\ep>0$ and a sequence $x_k \in \R^d$ 
	such that $\|x_k\| \to \infty$ and $\|\ph(t,x_k)\| \ge \ep$.
	Since $\Ph^\vee(s)$ is a diffeomorphism of $\R^d$, for all $s \in I$,
	\[
		\sup_{x \in \R^d} \|u(s,x + \ph(s,x))\| = \sup_{y \in \R^d} \|u(s,y)\| = \|u^\vee(s)\|_0
	\]
	and so the dominated convergence theorem implies that 
	\[
		\|\ph(t,x_k)\|  \le \int_0^t \|u(s,x_k + \ph(s,x_k))\| \, ds \to 0,
	\]
	because $\|x_k + \ph(s,x_k)\| \to \infty$ as $k \to \infty$ and $\|u(s,x)\| \to 0$ as $\|x\| \to \infty$, 
	for each $s \in I$; a contradiction.
	To see that $\|d_x^k \ph(t,x)\| \to 0$ as $\|x \| \to \infty$, for $1 \le k \le n$, we argue similarly:
		Since $s \mapsto \ph^\vee(s)$ is continuous into $C^{n,\be}_b(\R^d,\R^d)$, there is a constant 
		$C$ such that $\sup_{s \in I} \|\ph^\vee(s)\|_{n,\be} < C$. Thus Fa\`a di Bruno's formula \eqref{eq:FaadiBruno}
		implies that $\sup_{x \in \R^d} \|d_x^k (u(s,x + \ph(s,x)))\|$ is bounded above by $\|u(s)\|_k$ times a constant 
		(independent of $s$ and $x$) for all $s \in I$. 
		Then the dominated convergence theorem implies the assertion as before. 
\end{proof}

\begin{theorem} \label{thm:Trouve2}
	Let $n \in \N_{\ge 1}$ and $\be \in (0,1]$. Then 
	\begin{equation} \label{eq:Trouve2}
		\cG_{n,\be}(\R^d) = \cD^{n,\be}(\R^d)_0,
	\end{equation}
	where $\cD^{n,\be}(\R^d)_0$ denotes the connected component of the identity in $\cD^{n,\be}(\R^d)$. 	
\end{theorem}		

\begin{proof}
		The inclusion $\cG_{n,\be}(\R^d) \subseteq  \cD^{n,\be}(\R^d)_0$ follows from \Cref{thm:Trouve1}. 

		Let us prove $\cD^{n,\be}(\R^d)_0 \subseteq \cG_{n,\be}(\R^d)$.
		Since $\cD^{n,\be}(\R^d)_0$ is connected and locally path-connected, it is path-connected, 
		and each $\Ph \in \cD^{n,\be}(\R^d)_0$ can be connected by a polygon with the identity. 

		Let $\Ph = \Id +\ph \in \cD^{n, \be}(\R^d)_0$ be such that 
		$\ga(t):=(1-t)\Id + t\Ph \in \cD^{n,\be}(\R^d)$ for all $t \in I$. 
		Then $\ga(t)(x) = x + t \ph(x)$, and
		\[
			u(t,x) := (\ga'(t) \o \ga(t)^{-1})(x) = \ph (\ga(t)^{-1}(x))  
		\]
		is a time-dependent vector field such that: 
		\begin{itemize}
			\item $u(t,\cdot) \in C^{n,\be}_0(\R^d,\R^d)$ for all $t \in I$, since 
			$\cD^{n,\be}(\R^d)$ is a group, by \Cref{thm:group}.
			\item $u(\cdot, x)$ is a Borel function for every $x \in \R^n$; indeed, if $\ga(t)^{-1}(x) =: x + \ta(t,x)$ 
			then $\ta$ satisfies the implicit equation 
			\[
				\ta(t,x) + t \ph(x+\ta(t,x)) =0,
			\] 
			and is $C^n$ by the implicit function theorem.
			\item We have $\int_0^1 \|u(t,\cdot)\|_{n,\be} \,dt < \infty$, since inversion is locally bounded 
			on $\cD^{n,\be}(\R^d)$ and left translation maps bounded sets to bounded sets, see \Cref{comp1} and \Cref{invcl}. 
		\end{itemize}
		That means that $u \in \fX_{n,\be}(I,\R^d)$ and hence    
		$\Ph \in \mathcal{G}_{n,\be}(\R^d)$.

		Suppose we are given a polygon in $\mathcal{D}^{n,\be}(\R^d)$ with vertices $\Id,\Phi_1, \dots, \Phi_n$. 
		Then $\Ph_1 \in \mathcal{G}_{n,\be}(\R^d)$, by the previous paragraph. 
		Consider the line segment $\ga$ connecting $\Ph_1$ and $\Ph_2$. 
		Then $t\mapsto \gamma(t)\circ \Ph_1^{-1}$ connects $\Id$ with $\Phi_2\circ \Phi_1^{-1}$. 
		So, by the above, 
		$\Phi_2\circ \Phi_1^{-1} \in \mathcal{G}_{n,\be}(\R^d)$ and hence $\Phi_2 \in \mathcal{G}_{n,\be}(\R^d)$,
		since $\mathcal{G}_{n,\be}(\R^d)$ is a group.
		By iteration all vertices $\Ph_j$ belong to $\cG_{n,\be}(\R^d)$.	
\end{proof}		

\begin{remark} \label{rem:Trouve}
	Analyzing the proof one finds that the identity \eqref{eq:Trouve2} still holds if in the definition of the 
	Trouv\'e group we restrict to $u \in \fX_{n,\be}(I,\R^d)$ which are piecewise $C^n$ in time $t$.
	(Discontinuities in $t$ guarantee that $\cG_{n,\be}(\R^d)$ is a group.)  
\end{remark}

\subsection{The Trouv\'e group of \texorpdfstring{$C^{n,\be\pm}_0(\R^d,\R^d)$}{Cbeta-0}} \label{ssec:pointwisepm}

We define \emph{pointwise time-dependent $C^{n,\be\pm}_0$-vector fields} to be the elements of
\[
	\fX_{n,\be-}(I,\R^d) := \bigcap_{\al \in (0,\be)} \fX_{n,\al}(I,\R^d), 
	\quad \fX_{n,\be+}(I,\R^d) := \bigcup_{\al \in (\be,1)} \fX_{n,\al}(I,\R^d),
\]
respectively, 
and the corresponding Trouv\'e groups by
\[
	\cG_{n,\be\pm}(\R^d) := \big\{ \Ph_u^\vee(1) : u \in \fX_{n,\be\pm}(I,\R^d) \big\}.
\]

	\begin{theorem}
	 	\label{goal}
	 	Let $n \in \N_{\ge 1}$. 
	 	For $\be \in (0,1]$, $C^{n,\be-}_0$ is ODE closed and 
		\begin{equation}
			\cG_{n,\be-}(\R^d) = \cD^{n,\be-}(\R^d)_0,
		\end{equation}
		and, for $\be \in [0,1)$, $C^{n,\be+}_0$ is ODE closed and
		\begin{equation} \label{eq:goal2}
			\cG_{n,\be+}(\R^d) = \cD^{n,\be+}(\R^d)_0,
		\end{equation}
		In particular, $\cG_{n,\be\pm}(\R^d)$ has a $C^{0,\om}$ Lie group structure, 
		for every slowly vanishing modulus of continuity $\om$.
		Moreover, $\cG_{n,\be-}(\R^d)$ has a topological group structure and a half-Lie group structure.
	 \end{theorem}
	 
	\begin{proof}
		This follows from \Cref{thm:Trouve1}, \Cref{thm:Trouve2}, and \Cref{thm:C0omegaLie}. 
		To see, e.g., \eqref{eq:goal2}, note that 
		\[
			\cG_{n,\be+}(\R^d) 
			= \bigcup_{\al>\be} \cG_{n,\al}(\R^d) 
			= \bigcup_{\al>\be} \cD^{n,\al}(\R^d)_0 
		\]
		is path-connected in $\cD^{n,\be+}(\R^d) =  \bigcup_{\al>\be} \cD^{n,\al}(\R^d)$ and thus contained in 
		$\cD^{n,\be+}(\R^n)_0$. 
		The inclusion $\cD^{n,\be+}(\R^d)_0 \subseteq \cG_{n,\be+}(\R^d)$ follows from the proof of  
		\Cref{thm:Trouve2}: the line segment $\ga$ factors to some step of the inductive limit defining $\cD^{n,\be+}(\R^d)$, 
		by \Cref{lem:cpreg}.
	\end{proof}

Clearly, \Cref{rem:Trouve} also applies in this situation.

\medskip

Let us summarize the results of this section in \Cref{table:ode}.

{\small
\begin{table}[H]
\centering
\begin{tabular}{@{}lcccc@{}}
  & ODE closed & Trouv\'e group  \\ \midrule
 $C^{n,\be}_0(\R^d,\R^d)$ & yes & $\cD^{n,\be}(\R^d)_0$    \\ \midrule
 $C^{n,\be-}_0(\R^d,\R^d)$ & yes & $\cD^{n,\be-}(\R^d)_0$   \\
 \midrule
 $C^{n,\be+}_0(\R^d,\R^d)$ & yes & $\cD^{n,\be+}(\R^d)_0$  \\ 
\end{tabular}
\caption{\label{table:ode}\small{Here $n \in \N_{\ge 1}$. In the first two rows 
$\be \in (0,1]$, in the third row $\be \in [0,1)$.}}
\end{table}
}

	\section{Continuity of the flow map} \label{sec:flows}

Let $n \in \N_{\ge 1}$ and $\be \in (0,1]$.
By the results of the last section every $u \in \fX_{n,\be}(I,\R^d)$ (thus every $u^\vee \in L^1(I,C^{n,\be}_0(\R^d,\R^d))$) 
has a unique pointwise flow $\Ph$ with $\Ph^\vee \in C(I,\cD^{n,\be}(\R^d))$.
The goal of this section is to show the following: 
\begin{enumerate}
	\item If $u^\vee \in L^1(I,C^{n,\be}_0(\R^d,\R^d))$ and $\al < \be$, then $\Ph^\vee$ is the unique 
	strong $\cD^{n,\al}$-flow of $u$, i.e., 
	\begin{equation} \label{eq:Banachvalued}
		\Ph^\vee(t) = \Id + \int_0^t u^\vee(s) \o \Ph^\vee (s)\,ds, \quad  t \in I, 
	\end{equation}
	in $\cD^{n,\al}(\R^d)$.
	\item The flow map $L^1(I,C^{n,\be}_0(\R^d,\R^d)) \to C(I,\cD^{n,\al}(\R^d))$, $u^\vee \mapsto \Ph^\vee$, 
	is bounded for all $n \ge 1$ and continuous, even $C^{0,\be-\al}$, if $n \ge 2$.
\end{enumerate}
We recall that the Bochner integral in \eqref{eq:Banachvalued} might not exist if $\al = \be$, because  
strong measurability of $s \mapsto u^\vee(s) \o \Ph^\vee (s)$ may fail if $u^\vee(s)_\star$ is not continuous; 
see \Cref{rem:5.2} below.

\subsection{Existence of the strong \texorpdfstring{$\cD^{n,\alpha}$}{Dnalpha}-flow} \label{ssec:strongflow}

	First we show that the Bochner integral in \eqref{eq:Banachvalued} exists if $\be > \al$. 
	
	\begin{lemma}
		\label{point}
		Let $n \in \N_{\ge 1}$ and $0< \al < \be \le 1$.
		Let $u\in L^1(I, C^{n,\beta}_0(\R^d,\R^d))$
		and let $\Ph:I \rightarrow \mathcal{D}^{n,\alpha}(\R^d)$ be continuous. Then:
		\begin{itemize}
			\item[(i)] 
			The mapping $I \times C^{n,\al}_0(\R^d,\R^d) \to C^{n,\al}_0(\R^d,\R^d)$, 
			$(t,\ph) \mapsto u(t)\circ (\Id + \ph)$, has the Carath\'eodory property.
			The function $t\mapsto u(t)\circ \Ph(t)$ belongs to $L^1(I, C^{n,\al}_0(\R^d,\R^d))$.
			\item[(ii)] If $\ev_x\Phi(t)=: \Phi^\wedge(t,x)$ is the pointwise flow of $u$ for the initial condition 
			$\Phi^\wedge(0,x)=x$ for all $x$, then $\Phi$ is the strong $\mathcal{D}^{n,\alpha}$-flow of $u$.
		\end{itemize}
	\end{lemma}
	
	\begin{proof}
		(i) That $(t,\ph) \mapsto u(t)\circ (\Id + \ph)$ has the Carath\'eodory property follows easily from 
		\Cref{comp1} and \Cref{6.2}. Together with continuity of $\Phi$, an application of \cite[Lemma 2.2]{AulbachWanner96} yields strong 
		measurability of $t \mapsto u(t)\circ \Ph(t)$. Integrability follows from \Cref{comp1}.
			
			(ii) Recall that $\ev_x$ is continuous and linear on $C^{n,\alpha}_0(\R^d,\R^d)$ and thus commutes with the Bochner integral. Observe also that $s\mapsto u(s) \circ \Phi(s)$ is Bochner integrable in $C^{n,\al}_0(\R^d,\R^d)$, by (i). 
			Since $\Phi^\wedge(t,x)$ is the pointwise flow, we have
			\begin{align*} 
				\ev_x \Ph(t) = \Phi^\wedge(t,x)=x+\int_0^t u(s)(\Phi^\wedge(s,x)) \, ds
				= \ev_x \Big( \Id + \int_0^tu(s)\circ \Phi(s) \, ds  \Big).
			\end{align*}
			Since the family of evaluation maps is point separating on $C^{n,\alpha}_0(\R^d,\R^d)$, we are done.
	\end{proof}
	
	\begin{remark} \label{rem:5.2}
		In general, the function $t\mapsto u(t)\circ \Ph(t)$ cannot belong to $L^1(I, C^{n,\be}_0(\R^d,\R^d))$, 
		even if $u$ is a constant and  
		$\Ph$ is continuous into $\cD^{N,1}(\R^d)$ for all $N \ge n$. Indeed: 
		Let $\chi \in C^\infty_c(\R)$ be $1$ on $[-1,1]$, and let $\ps(x):= x^n |x|^\beta \ch(x)$ be the 
	    $C^{n,\be}_0$-function from the proof of \Cref{disc} and \Cref{example:strongmeasurability}.
	    Taking $u(t) := \ps$ and $\Ph(t)(x) := x + t \ch(x)$ we get 
	    $(u(t) \o \Ph(t))(x) = \ps(x+t)$ if $x \in [-1,1]$. So \Cref{example:strongmeasurability} shows that 
	    $t\mapsto u(t)\circ \Ph(t)$ is not strongly measurable.
	\end{remark}

	\begin{theorem}
		\label{odesolv}
		Let $n \in \N_{\ge 1}$ and $0<\al < \be \le 1$.
		Let $u \in L^1(I, C^{n,\beta}_0(\R^d,\R^d))$. Then $u$ has a unique strong $\mathcal{D}^{n,\alpha}$-flow $\Ph$. 
	\end{theorem}

	\begin{proof}
		By \Cref{thm:Trouve1}, $u$ has a unique pointwise flow $\Ph : I \times \R^d \to \R^d$ such that 
		$\Ph^\vee \in C(I,\cD^{n,\be}(\R^d))$. 
		For $\al < \be$, we have $\Ph^\vee \in C(I,\cD^{n,\al}(\R^d))$. 
		By \Cref{point}, $\Ph^\vee$ is the unique strong $\cD^{n,\al}$-flow of $u$, i.e., it satisfies \eqref{eq:Banachvalued}.
	\end{proof}

	\begin{remark}
		One can use Carath\'eodory's solution theory for ODEs on Banach spaces 
		which are Bochner integrable in time (cf.\ \Cref{carath}) to give an alternative 
		proof which, however, does not work for $n=1$!
		Indeed, using \Cref{point}, \Cref{Lipschitz}, and \Cref{cara} one can show that $u$ has a unique 
		strong $\cD^{n-1,\al}$-flow $\Ph$. That the flow $\Ph = \Id + \ph$ is actually strongly $\cD^{n,\al}$-valued follows from 
		the observation that $d\ph$ satisfies the linear ODE  
		\begin{align*} 
			d \ph(t) &= d\int_0^t u(s)\circ (\Id+\ph(s))\, ds 
			=\int_0^t du(s)\circ(\text {Id}+\ph(s))\cdot (\mathbb{1}+d\ph)(s) \, ds,
		\end{align*}
		and from \Cref{lincara}.
	\end{remark}

\subsection{Continuity of the flow map}

First we prove that the flow map is bounded.

	\begin{proposition} \label{odebounded}
		Let $n \in \N_{\ge 1}$ and $\be \in (0,1]$.
		The flow map $L^1(I,C^{n,\be}_0(\R^d,\R^d)) \to C(I,\cD^{n,\be}(\R^d))$, $u \mapsto \Ph$, 
		is bounded.
	\end{proposition}

	In the proof we use only pointwise estimates; thus the result still holds if $u^\wedge$ varies in 
	$\fX_{n,\be}(I,\R^d)$ endowed with the norm $\|u\|_{L^1(I,C^{n,\be}_b)}$.  

	\begin{proof}
		We first claim that $u \mapsto \ph:= \Ph -\Id$ is bounded into $C(I,C^n_b(\R^d,\R^d))$. 
We proceed by induction on $n$.
For simplicity of notation we simply write $\ph(t,x)$ instead of $\ph^\wedge(t,x)$, etc.
Clearly 
\[
		\|\ph(t,x)\| \le \int_0^t \|u(s,\Ph(s,x))\| \,ds \le \|u\|_{L^1(I,C^{n,\be}_b)}
\]
and hence 
\[
	\|\ph\|_{C(I,C^0_b)} \le \|u\|_{L^1(I,C^{n,\be}_b)}.
\]
Assume that $u \mapsto \ph$ is bounded into $C(I,C^{n-1}_b(\R^d,\R^d))$.
		By Fa\`a di Bruno's formula \eqref{eq:FaadiBruno}, 
		\begin{align*} 
				d_x^n(u(s)\circ \Ph(s))(x)
				&= d_x u(s)(\Ph(s,x))(d_x^n \Ph(s,x))\\
				& + \on{sym} \sum_{l=2}^{n} \sum_{\ga \in \Ga(l,n)}
				 c_{\ga} u(s)^{(l)}(\Ph(s,x))\big(d_x^{\ga_1}\Ph(s,x), \dots, d_x^{\ga_l} \Ph(s,x)\big).
		\end{align*}
		Hence
		\begin{align*}
				\MoveEqLeft 
				\|d_x^n(u(s)\circ \Ph(s))\|_{0} \\
				&\le   
				\|u(s)\|_{1} [\Ph(s)]_n + \sum_{l=2}^{n} \sum_{\ga \in \Ga(l,n)}
				 c_{\ga} \|u(s)\|_{l} [\Ph(s)]_{\ga_1} \cdots [\Ph(s)]_{\ga_l}
		\end{align*}
		and so, by induction hypothesis and since $[\Ph(s)]_n \le 1 + [\ph(s)]_n$, 
		\begin{align*}
			[\ph(t)]_{n} &\le \int_0^t \|d_x^n(u(s)\circ \Ph(s))\|_{0} \,ds
			\\
			&
			\le \int_0^t \|u(s)\|_{1} [\ph(s)]_{n} \,ds +  C \|u\|_{L^1(I,C^{n,\be}_b)}.
		\end{align*}
Gronwall's lemma implies that    
\[
	\|\ph\|_{C(I,C^n_b)} \le C \|u\|_{L^1(I,C^{n,\be}_b)} \exp (\|u\|_{L^1(I,C^{n,\be}_b)}), 
\]
and the claim is proved. 

It remains to show that $u \mapsto \ph$ is bounded into $C(I,C^{n,\al}_b(\R^d,\R^d))$. To this end consider 
\begin{align*}
				\MoveEqLeft 
				d_x^n(u(s)\circ \Ph(s))(x) - d_x^n(u(s)\circ \Ph(s))(y)
				\\
				&= \!\on{sym}\! \sum_{l=1}^{n} \sum_{\ga \in \Ga(l,n)}
				 c_{\ga} \big(u(s)^{(l)}(\Ph(s,x))(A^{\ga,l}(x,y)) - u(s)^{(l)}(\Ph(s,y))(A^{\ga,0}(x,y)) \big),
\end{align*}
where 
\[
	A^{\ga,j}= A^{\ga,j}(x,y) :=  \big(d_x^{\ga_1}\Ph(s,x), \dots, d_x^{\ga_j}\Ph(s,x), 
	d_x^{\ga_{j+1}}\Ph(s,y),  \dots, d_x^{\ga_l}\Ph(s,y) \big).
\]
Then
\begin{align*}
				\MoveEqLeft
				\big\| u(s)^{(l)}(\Ph(s,x))(A^{\gamma,l}) -  u(s)^{(l)}(\Ph(s,y))(A^{\gamma,0})\big\|_{L_n}\\
			&\leq \big\| u(s)^{(l)}(\Ph(s,x))(A^{\gamma,l}) - u(s)^{(l)}(\Ph(s,y))(A^{\gamma,l})\big\|_{L_n}\\
				& \quad + 
				\sum_{k=1}^{l} \big\| u(s)^{(l)}(\Ph(s,y))(A^{\gamma,k}) 
				- u(s)^{(l)}(\Ph(s,y))(A^{\gamma,k-1})\big\|_{L_n}.
		\end{align*}
For the first summand		
\begin{align*}
	\MoveEqLeft
	\big\| u(s)^{(l)}(\Ph(s,x))(A^{\gamma,l}) - u(s)^{(l)}(\Ph(s,y))(A^{\gamma,l})\big\|_{L_n}
	\\
	&\le \big\| u(s)^{(l)}(\Ph(s,x)) - u(s)^{(l)}(\Ph(s,y))\big\|_{L_n} (1 + \|\ph(s)\|_{n})^n
	\\
	&\le \begin{cases}
		\| u(s)\|_{n} [\Ph(s)]_{1} \| x- y\|  (1 + \|\ph(s)\|_{n})^n & \text{ if } l < n, \\ 
		\| u(s)\|_{n,\be} [\Ph(s)]_{1}^\be \| x- y\|^\be  (1 + \|\ph(s)\|_{n})^n & \text{ if } l = n.
	\end{cases} 
\end{align*}
For the other summands 
		observe that
		\begin{align*}
			&u(s)^{(l)} (\Ph(s,y))(A^{\gamma,k}) 
				- u(s)^{(l)}(\Ph(s,y))(A^{\gamma,k-1})\\
			&=  u(s)^{(l)}(\Ph(s,y))\big(  \dots, d_x^{\ga_{k-1}}\Ph(s,x),
			 d_x^{\ga_k} \Ph(s,x) - d_x^{\ga_k}\Ph(s,y) , d_x^{\ga_{k+1}}\Ph(s,y),\dots \big),
		\end{align*}
		whence, if $l \ge 2$ and hence $\ga_k \le n-1$, 
		\begin{align*}
			\MoveEqLeft
			\big\|u(s)^{(l)}(\Ph(s,y))(A^{\gamma,k}) 
				- u(s)^{(l)}(\Ph(s,y))(A^{\gamma,k-1})\big\|_{L_n}\\
			&\le \| u(s)\|_{n} (1 + \|\ph(s)\|_{n})^{n-1}
			   \|\ph(s)\|_{n}\|x-y\|.
		\end{align*}
For $l =1$, we have 		
\begin{align*}
			\MoveEqLeft
			\big\|du(s)(\Ph(s,y))(d_x^n\Ph(s,x)) 
				- du(s)(\Ph(s,y))(d_x^n\Ph(s,y))\big\|_{L_n}
				\\
			&\le \| u(s)\|_{n} 
			   \|\ph(s)\|_{n,\be} \|x-y\|^\be.
\end{align*}
These estimates, together with the fact that $u \mapsto \ph$ is bounded into $C(I,C^n_b(\R^d,\R^d))$, imply
\begin{align*}
			\|\ph(t)\|_{n,\be} &\le \int_0^t \|u(s)\circ \Ph(s)\|_{n,\be} \,ds
			\\
			&\le C \int_0^t \|u(s)\|_{n,\be} \|\ph(s)\|_{n,\be} \,ds +  C \|u\|_{L^1(I,C^{n,\be}_b)},
		\end{align*}
and Gronwall's inequality yields the assertion.
\end{proof}

	\begin{theorem} \label{thm:flowcontinuous}
		Let $n \in \N_{\ge 2}$ and $0 < \al < \be \le 1$.
		Then the flow map $L^1(I,C^{n,\be}_0(\R^d,\R^d)) \to C(I,\cD^{n,\al}(\R^d))$, $u \mapsto \Ph$, 
		is continuous, even $C^{0,\be-\al}$. 	
	\end{theorem}

	We do not know if the theorem also holds for $n =1$ or for $\al = \be$.

	\begin{proof}
		Fix $u_0 \in L^1(I,C^{n,\be}_0(\R^d,\R^d))$ and let $u, v \in L^1(I,C^{n,\be}_0(\R^d,\R^d))$ be in the ball with radius $\de>0$ and center $u_0$ in $L^1(I,C^{n,\be}_0(\R^d,\R^d))$.  
		Consider the corresponding flows
		$\Ph=\Id + \ph,\Ps = \Id + \ps \in C(I,\cD^{n,\al}(\R^d))$. 
		By \Cref{odebounded}, there is a constant $C= C(u_0,\de)>0$ such that 
		\begin{equation*}
			\|\ph\|_{C(I,C^{n,\al}_b)} \le C,\quad \|\ps\|_{C(I,C^{n,\al}_b)} \le C.
		\end{equation*}
		By \Cref{Lipschitz}, \Cref{comp1}, and \Cref{odesolv}, 
		\begin{align*}
			\MoveEqLeft
			\|\ph(t) - \ps(t)\|_{n-1,\al} 
			\le \int_0^t \|u(s)\circ\Ph(s) - v(s)\circ\Ps(s)\|_{n-1,\al} \,ds
			\\
			& \le \int_0^t \|u(s)\circ\Ph(s) - u(s)\circ\Ps(s)\|_{n-1,\al} + \|(u(s) - v(s))\circ\Ps(s)\|_{n-1,\al} \,ds \\
			& \le C_1 \int_0^t \|u(s)\|_{n,\be} \|\ph(s) - \ps(s)\|_{n-1,\al} 	
			+ \|u(s) - v(s)\|_{n-1,\al} \,ds	
		\end{align*}
		and so, by Gronwall's lemma (and \Cref{inclusion}),
		\begin{align} \label{Gronwall}
		 	\|\ph(t) - \ps(t)\|_{n-1,\al} 
		 	&\le C_1 \|u - v\|_{L^1(I,C^{n,\be}_b)} \exp (C_1 \|u\|_{L^1(I,C^{n,\be}_b)})
		 	\notag \\
		 	&=:
		 	C_2 \|u - v\|_{L^1(I,C^{n,\be}_b)}.
		 \end{align} 
		This proves that $u \mapsto \Ph$ is continuous into $C(I,\cD^{n-1,\al}(\R^d))$.  
		Applying $d =d_x$,
		\begin{align*}
			\MoveEqLeft
		 	\|d\ph(t) - d\ps(t)\|_{n-1,\al} 
			\le \int_0^t \|(d u(s)\circ\Ph(s))d\Ph(s) -  (dv(s)\circ\Ps(s)) d\Ps(s)\|_{n-1,\al} \,ds 
			\\
			&\le \int_0^t \|(d u(s)\circ\Ph(s))(d\Ph(s) -  d\Ps(s))\|_{n-1,\al} \,ds 
			\\
			&\quad + 
			\int_0^t \|(d u(s)\circ\Ph(s) - dv(s)\circ\Ps(s))   d\Ps(s))\|_{n-1,\al} \,ds.
		 \end{align*} 
		 By \Cref{bilinear} and \Cref{comp1},
		 \begin{align*}
		 	\MoveEqLeft
		 	\|(d u(s)\circ\Ph(s))(d\Ph(s) -  d\Ps(s))\|_{n-1,\al} 
		 	\\
		 	&\le 2^n 
		 	\|d u(s)\circ\Ph(s)\|_{n-1,\al} \|d\Ph(s) -  d\Ps(s)\|_{n-1,\al}
		 	\\
		 	&\le 2^n M 
		 	\|d u(s)\|_{n-1,\al} (1+\|\ph(s)\|_{n-1,\al})^n \|d\Ph(s) -  d\Ps(s)\|_{n-1,\al}
		 	\\
		 	&\le C_3 
		 	\|u(s)\|_{n,\be}  \|d\Ph(s) -  d\Ps(s)\|_{n-1,\al} 
		 	\intertext{and} 
		 	\MoveEqLeft
		 	\|(d u(s)\circ\Ph(s) - dv(s)\circ\Ps(s))   d\Ps(s)\|_{n-1,\al} 
		 	\\
		 	&\le 
		 	2^n\|d u(s)\circ\Ph(s) - dv(s)\circ\Ps(s) \|_{n-1,\al} \|  d\Ps(s)\|_{n-1,\al}
		 	\\
		 	&\le 
		 	C_4\|d u(s)\circ\Ph(s) - dv(s)\circ\Ps(s) \|_{n-1,\al}.
		 \end{align*}
		 By \Cref{comp1}, \Cref{6.2}, and \eqref{Gronwall}, 
		 \begin{align*}
		 	&\| d u(s)\circ\Ph(s) - dv(s)\circ\Ps(s) \|_{n-1,\al}
		 	\\
		 	&\le
		 	\|(d u(s) - dv(s))\circ\Ph(s) \|_{n-1,\al} + \|d v(s)\circ\Ph(s) - dv(s)\circ\Ps(s) \|_{n-1,\al}
		 	\\
		 	&\le
		 	M \big(
		 	\|d u(s)- dv(s)\|_{n-1,\al} (1+\|\ph(s)\|_{n-1,\al})^n +  \|v(s)\|_{n,\be} 
		 	\|\ph(s) - \ps(s)\|^{\be-\al}_{n-1,\al} \big)
		 	\\
		 	&\le 
		 	C_5\big(\|u(s)- v(s)\|_{n,\be} +  \|v(s)\|_{n,\be} 
		 			 	\|\ph(s) - \ps(s)\|^{\be-\al}_{n-1,\al}\big). 
		 \end{align*}
		 Together with \eqref{Gronwall} this gives 
		 \begin{align*}
		 	\MoveEqLeft
		 	\int_0^t \|d u(s)\circ\Ph(s) - dv(s)\circ\Ps(s) \|_{n-1,\al}\, ds
		 	\\
		 	&
		 	\le
		 	C_6 \big(\|u - v\|_{L^1(I,C^{n,\be}_b)} +   
		 	 \|u - v\|^{\be-\al}_{L^1(I,C^{n,\be}_b)}   \big)
		 	 \\
		 	&
		 	\le
		 	C_7   
		 	 \|u - v\|^{\be-\al}_{L^1(I,C^{n,\be}_b)},
		 \end{align*}
		 provided that $\|u - v\|_{L^1(I,C^{n,\be}_b)} \le 1$.
		 Consequently,
		 \begin{align*}
		 	\|d\ph(t) - d\ps(t)\|_{n-1,\al}  
			&\le C_3 \int_0^t  
		 	\|u(s)\|_{n,\be}  \|d\ph(s) -  d\ps(s)\|_{n-1,\al} \,ds 
			\\
			&\quad +
			C_7   
		 	 \|u - v\|^{\be-\al}_{L^1(I,C^{n,\be}_b)}.
		 \end{align*}
		 Then Gronwall's inequality implies 
		 \begin{align*}
		 	\|d\ph(t) - d\ps(t)\|_{n-1,\al} 
		 	&\le C_7   
		 	 \|u - v\|^{\be-\al}_{L^1(I,C^{n,\be}_b)} \exp (C_3  \|u\|_{L^1(I,C^{n,\be}_b)})
		 	 \\
		 	&\le C_8   
		 	 \|u - v\|^{\be-\al}_{L^1(I,C^{n,\be}_b)},
		 \end{align*}
		 for all $t \in I$, and the assertion follows.	
	\end{proof}

	\subsection{Flows of strong time-dependent \texorpdfstring{$C^{n,\beta-}_0$}{Cnbetapm}-vector fields}
\label{ssec:tdepvf}

Let $n \in \N_{\ge 1}$.

By a \emph{strong time-dependent $C^{n,\beta-}_0$-vector field}, for $\be \in (0,1]$, 
we mean a 
function $u : I \to C^{n,\beta-}_0(\R^d,\R^d)$ 
such that $u \in L^1(I,C^{n,\al}_0(\R^d,\R^d))$ for all $\al < \be$. We denote the space of all strong 
time-dependent 
$C^{n,\beta-}_0$-vector fields by $L^1(I,C^{n,\be-}_0(\R^d,\R^d))$ and equip it with the fundamental 
system of seminorms $\{\|\cdot\|_{L^1(I,C^{n,\al}_b)} :   \al < \be\}$.

Clearly,
for every strong time-dependent $C^{n,\be-}_0$-vector field $u$, $u^\wedge$  
is a pointwise time-dependent $C^{n,\be-}_0$-vector field (as defined in \Cref{ssec:pointwisepm}); 
the converse is not true in general.

By \Cref{odebounded}, the flow map $L^1(I,C^{n,\be-}_0(\R^d,\R^d)) \to C(I,\cD^{n,\be-}(\R^d))$ is bounded, 
for all $n \in \N_{\ge 1}$, $\be \in (0,1]$.

\begin{theorem}
	Let $n \in \N_{\ge 2}$. For $\be \in (0,1]$,
	the flow map $L^1(I,C^{n,\be-}_0(\R^d,\R^d)) \to C(I,\cD^{n,\be-}(\R^d))$, $u \mapsto \Ph$, 
	is continuous and $C^{0,\om}$, for any slowly vanishing modulus of continuity $\om$. 
\end{theorem}

\begin{proof}
	This is immediate from \Cref{thm:flowcontinuous} and from the estimates in its proof.
\end{proof}

\begin{remark}
	One could define \emph{strong time-dependent $C^{n,\beta+}_0$-vector fields}, for $\be \in [0,1)$, 
	to be the elements of the (LB)-space $\bigcup_{\al \in (\be,1)} L^1(I,C^{n,\al}_0(\R^d,\R^d))$.
	Then, by \Cref{thm:Trouve1}, we have a flow map
	\[
		\bigcup_{\al \in (\be,1)} L^1(I,C^{n,\al}_0(\R^d,\R^d)) \to \bigcup_{\al \in (\be,1)} C(I,C^{n,\al}_0(\R^d,\R^d)) 
		\subseteq 
		C(I,C^{n,\be+}_0(\R^d,\R^d)). 
	\]
	Note that this is not clear for $u \in L^1(I,C^{n,\be+}_0(\R^d,\R^d))$, since such $u$ may not factor to some step 
	in the inductive limit defining $C^{n,\be+}_0(\R^d,\R^d)$.  
	Is this flow map $C^{0,\om}$, for slowly vanishing moduli of continuity $\om$?
	This would follow from \Cref{thm:flowcontinuous} if the (LB)-space $\bigcup_{\al \in (\be,1)} L^1(I,C^{n,\al}_0(\R^d,\R^d))$ were regular. 
\end{remark}

\appendix

\section{Proofs for \texorpdfstring{\Cref{composition}}{Section 2.4}} \label{appendix}

\Cref{bilinear} is precisely \cite[4.2]{LlaveObaya99}.

\begin{proof}[Proof of \Cref{comp1}]
	We prove the assertion by induction on $m$. 
	First observe that $d(g\circ(\Id+f))= dg\circ(\Id+f) \cdot (\mathbb{1}+df)= dg\circ(\Id+f) + dg\circ(\Id+f)\cdot df$. We have
			\begin{align} \label{eq:Appendix1}
				\|dg(x+f(x)) -dg(y+f(y))\|_{L_1} &\leq \|dg\|_{0,\alpha}\|x-y+f(x)-f(y)\|^\alpha \notag \\
				&\leq \|g\|_{1,\alpha}(1+\|f\|_{1})^\alpha \|x-y\|^\alpha,
			\end{align}
			and
			\begin{align*}
				\MoveEqLeft 
				\|dg(x+ f(x))\cdot df(x) - dg(y+ f(y))\cdot df(y)\|_{L_1}
				\\
				&\le \|dg(x+ f(x))\cdot df(x) - dg(x+ f(x))\cdot df(y)\|_{L_1}\\
				& \qquad +  \|dg(x+ f(x))\cdot df(y) - dg(y+ f(y))\cdot df(y)\|_{L_1}\\
				&\le  \|dg(x+f(x))\|_{L_1} \|df(x)-df(y)\|_{L_1} \\
				& \qquad + \|dg(x+f(x))-dg(y+f(y))\|_{L_1} \|df(y)\|_{L_1}\\
				&\leq \|g\|_{1,\alpha} \|f\|_{1,\alpha} \|x-y\|^\alpha + \|g\|_{1,\alpha}(1+\|f\|_{1})^\alpha \|x-y\|^\alpha \|f\|_{1,\alpha}.
			\end{align*}
			Thus,
			\begin{equation*}
				\|dg\circ (\text {Id}+ f)\cdot df\|_{0,\alpha} \le 2 \|g\|_{1,\alpha}(1+\|f\|_{1,\alpha})^{1+\alpha},
			\end{equation*}
			and since the same bound is trivially also valid for $\|g\circ(\Id+f)\|_0$, the case $m = 1$ is proved.

			Now assume the statement holds for $m-1$. Then 
			\begin{align*}
				\|d(g\circ(\Id+f))\|_{m-1,\alpha} &\leq \|dg\circ(\Id+f) \|_{m-1,\alpha} + \|dg\circ(\Id+f)\cdot df \|_{m-1,\alpha}.
			\end{align*}
			The inductive assumption implies 
			\begin{equation*}
				\|dg\circ(\Id+f) \|_{m-1,\alpha} \leq M \|dg\|_{m-1,\alpha}(1+\|f\|_{m-1,\alpha})^{m-1+\alpha},
			\end{equation*}
			and using \Cref{bilinear}, we get 
			\begin{align*}
				\|dg\circ(\Id+f) \cdot df \|_{m-1,\alpha}&\leq 2^m \|dg\circ(\Id+f)\|_{m-1,\alpha} \cdot \|df \|_{m-1,\alpha}
			\end{align*}
			which now adds up to  \eqref{eq:comp1}.
\end{proof}

\begin{proof}[Proof of \Cref{6.2}]
	We proceed by induction on $m$.  First observe that we have
			\begin{equation*}
				\|g_\star(f_1)-g_\star(f_2)\|_{0}\leq \|g\|_{1}\|f_1-f_2\|_{0}.
			\end{equation*}
			Moreover, by \Cref{bilinear},
			\begin{align*}
				\MoveEqLeft \|d(g_\star(f_1)) - d(g_\star(f_2))\|_{0,\alpha}\\
				&= \|dg\circ(\Id+f_1)\cdot (\mathbb{1}+df_1) - 	dg\circ(\Id+f_2)\cdot (\mathbb{1}+df_2)\|_{0,\alpha}\\
				&= 	\|dg\circ(\Id+f_1)\cdot (df_1-df_2)  	-(dg\circ(\Id+f_2) - dg\circ(\Id+f_1))\cdot (\mathbb{1}+df_2)\|_{0,\alpha}\\
				&\leq 2\|dg\circ(\Id+f_1)\|_{0,\alpha}\|df_1-df_2\|_{0,\alpha}	\\
				&\quad + \|dg\circ(\Id+f_2) - dg\circ(\Id+f_1)\|_{0,\alpha}(1+2\|df_2\|_{0,\alpha}).
			\end{align*}
			As an intermediate step we use \Cref{exp} and \eqref{eq:Appendix1} to estimate
			\begin{align*}
				\MoveEqLeft \|dg\circ(\Id+f_2) - dg\circ(\Id+f_1)\|_{0,\alpha}\\
				&\leq \|dg\circ(\Id+f_2) - dg\circ(\Id+f_1)\|_{0}^{\frac{\beta-\alpha}{\beta}}\|dg\circ(\Id+f_2) - dg\circ(\Id+f_1)\|_{0,\beta}^{\frac{\alpha}{\beta}}\\
				&\leq (\|g\|_{1,\beta}\|f_1-f_2\|_{0}^\beta)^\frac{\beta-\alpha}{\beta} (\|dg\circ(\Id+f_1)\|_{0,\beta} + \|dg\circ(\Id+f_2)\|_{0,\beta})^\frac{\alpha}{\beta}\\
				&\leq (\|g\|_{1,\beta}\|f_1-f_2\|_{0}^\beta)^\frac{\beta-\alpha}{\beta} (\|g\|_{1,\beta} ((1+\|f_1\|_{1})^\beta +  (1+\|f_2\|_{1})^\beta))^\frac{\alpha}{\beta}\\
				&\leq \|g\|_{1,\beta} (2+\|f_1\|_{1}+ \|f_2\|_{1})\|f_1-f_2\|_{0}^{\beta - \alpha}.
			\end{align*}
			Consequently, if $R>0$, $f_1,f_2 \in B^{1,\al}(f_0,R)$, and hence 
			$\|f_1-f_2\|_{1,\alpha} \leq (1+2R) \|f_1-f_2\|_{1,\alpha}^{\beta -\alpha}$, then
			\begin{align*}
				\|d(g_\star(f_1)) - d(g_\star(f_2))\|_{0,\alpha}
				\leq M \|g\|_{1,\beta}\|f_1-f_2\|_{1,\alpha}^{\beta -\alpha}, 
			\end{align*}	
			where $M = M(\|f_0\|_{1,\al}, R)$, and hence 
			\begin{equation*}
				\|g_\star(f_1)-g_\star(f_2)\|_{1,\alpha} \leq M\|g\|_{1,\beta}\|f_1-f_2\|_{1,\alpha}^{\beta -\alpha}
			\end{equation*}
			which proves the case $m = 1$.

			Now assume we have already proven the desired result for $m-1$. Then, as in the case $m = 1$, we have
			\begin{align*}
				\MoveEqLeft \|d(g_\star(f_1)) - d(g_\star(f_2))\|_{m-1,\alpha}\\
				&\leq 2^m\|dg\circ(\Id+f_1)\|_{m-1,\alpha}\|df_1-df_2\|_{m-1,\alpha}	\\
				&\quad + \|dg\circ(\Id+f_2) - dg\circ(\Id+f_1)\|_{m-1,\alpha}(1+2^m\|df_2\|_{m-1,\alpha}).
			\end{align*}
			By the inductive assumption, 
			\begin{align*}
				\|dg\circ(\Id+f_2) - dg\circ(\Id+f_1)\|_{m-1,\alpha} 
				\leq M\|g\|_{m,\beta} \|f_1-f_2\|_{m-1,\alpha}^{\beta -\alpha}.
			\end{align*}
			Together with \Cref{comp1}, which makes it possible to extract $\|g\|_{m,\beta}$ from the term $\|dg\circ(\Id+f_1)\|_{m-1,\alpha}$, and using that $\|f_1-f_2\|_{m,\al} \leq (1+2R)\|f_1-f_2\|_{m,\al}^{\be-\al}$ for $f_1,f_2 \in B^{m,\al}(f_0,R)$ , we may conclude 
			\eqref{eq:6.2}.	
\end{proof}

\begin{proof}[Proof of \Cref{jcont}]
		This follows easily from \Cref{comp1}, \Cref{6.2}, and
		\[
		g_1 \o (\Id+f_1) - g_2 \o (\Id+f_2) = f_1^\star (g_1  - g_2) + (g_2)_\star (f_1) - (g_2)_\star (f_2). \qedhere
		\]
\end{proof}

\begin{proof}[Proof of \Cref{6.7}]
		By \Cref{6.2}, the mapping $(dg)_\star: C^{m,\alpha}_b(\mathbb{R}^d,\mathbb{R}^d) \rightarrow C^{m,\alpha}_b(\mathbb{R}^d,L(\mathbb{R}^d,\mathbb{R}^d))$, $\phi \mapsto dg\circ(\Id+\phi)$ is continuous. Consider the mapping 
		\begin{align*}
			l:&~ C^{m,\alpha}_b(\mathbb{R}^d, L(\mathbb{R}^d,\mathbb{R}^d)) \rightarrow L(C^{m,\alpha}_b(\mathbb{R}^d,\mathbb{R}^d), C^{m,\alpha}_b(\mathbb{R}^d,\mathbb{R}^d))\\
			&u \mapsto l(u)(\eta):=(x\mapsto u(x)(\eta(x)))
		\end{align*}
		which is continuous and linear. 
		We claim that $d(g_\star)$ exists and satisfies $d(g_\star)=l\circ (dg)_\star$. This implies the proposition.

		First note that for $\psi_0 \in C^{m,\alpha}_b(\R^d,\R^d)$ and $\phi \in C^{m,\alpha}_b(\mathbb{R}^d,\mathbb{R}^d)$, 
		\begin{equation*}
			(l\circ (dg)_\star)(\psi_0)(\phi)(x) =dg \circ(\Id+\psi_0)(x) \cdot \phi(x),
		\end{equation*}
		where $\cdot$ denotes the action of the linear map $dg \circ(\Id+\psi_0)(x) \in L(\mathbb{R}^d, \mathbb{R}^d)$ to the vector $\phi(x) \in \R^d$.

		Take $\phi \in C^{m,\alpha}_b(\mathbb{R}^d,\mathbb{R}^d)$ with $\|\phi\|_{m,\alpha} \leq 1$.
		By \Cref{6.2} (applied to $dg$), for $\ps_1 \in B^{m,\al}(\ps_0,1)$,  
		\begin{equation} \label{eq:6.2.1}
				\|(dg)_\star(\ps_0)-(dg)_\star(\ps_1)\|_{m,\al} \le M\|g\|_{m+1,\beta} \|\ps_0-\ps_1\|_{m,\al}^{\beta - \al},
		\end{equation}
		For $\varepsilon < 1$ we have $\psi_0 + \varepsilon \phi \in B^{m,\al}(\psi_0,1)$ for all 
		$\phi \in B^{m,\al}(0,1)$. Now, by \Cref{bilinear} and \eqref{eq:6.2.1},
		\begin{align*}
			\MoveEqLeft
			\frac{1}{\varepsilon} \|g_\star(\psi_0 + \varepsilon \phi) -g_\star(\psi_0) 
			- (l\circ (dg)_\star)(\psi_0)(\varepsilon \phi)\|_{m,\alpha}\\
			&= 	\frac{1}{\varepsilon} \|g\circ(\Id+\psi_0 + \varepsilon \phi) -g\circ(\Id + \psi_0) - \varepsilon (dg \circ (\Id + \psi_0))\cdot \phi\|_{m,\alpha}\\
			&= \Big\| \int_0^1 (dg\circ(\Id+\psi_0 + s\varepsilon \phi) -dg\circ(\Id + \psi_0)   )\cdot \phi \, ds \Big\|_{m,\alpha}\\
			&\leq \int_0^1 2\| dg\circ(\Id+\psi_0 + s\varepsilon \phi) -dg\circ(\Id + \psi_0)  \|_{m,\alpha}\|\phi\|_{m,\alpha} \, ds\\
			& \leq \int_0^1 2 M\|g\|_{m+1,\beta} \|\varepsilon s \phi\|_{m,\alpha}^{\beta - \alpha} \, ds\\
			&\leq 2M \|g\|_{m+1,\beta} \varepsilon^{\beta - \alpha}
		\end{align*}
		which tends to $0$ uniformly in $\phi \in B^{m,\al}(0,1)$ as $\varepsilon \rightarrow 0$. The claim is proved.
	\end{proof}

\begin{proof}[Proof of \Cref{Lipschitz}]
	Let  $\gamma(s):= (1-s)f_1 + s f_2$ for $s \in [0,1]$.
	Using \Cref{ftc} and \Cref{6.7}, we get 
	\begin{align*}
	g_\star(f_1) - g_\star(f_2) &= \int_0^1 \frac{d}{ds}(g_\star\circ \gamma) (s)\, ds
	=\int_0^1 d(g_\star) (\gamma(s)) \cdot  \gamma'(s) \,ds\\
	&=\int_0^1 dg\circ(\Id + \gamma(s)) \cdot (f_2-f_1) \,ds.
	\end{align*}
	Thus, by \Cref{bilinear} and \Cref{comp1}, 
	\begin{align*}
	\MoveEqLeft
	\|g_\star( f_1) - g_\star(f_2)\|_{m,\alpha} 
	\\
	&\leq \int_0^1 \|dg\circ(\Id + \gamma(s)) \cdot (f_2-f_1)\|_{m,\alpha}\,ds
	\\& 
	\leq \int_0^1 M \|dg\|_{m,\beta} (1 + \|\gamma(s)\|_{m,\alpha})^{m+1}  \|f_2-f_1\|_{m,\alpha}\, ds\\
	&\leq M \|g\|_{m+1,\beta} (1 + \max_{i=1,2}\|f_i\|_{m,\alpha})^{m+1}  \|f_2-f_1\|_{m,\alpha}.
	\qedhere
	\end{align*}
\end{proof}

\def\cprime{$'$}
\providecommand{\bysame}{\leavevmode\hbox to3em{\hrulefill}\thinspace}
\providecommand{\MR}{\relax\ifhmode\unskip\space\fi MR }
\providecommand{\MRhref}[2]{%
  \href{http://www.ams.org/mathscinet-getitem?mr=#1}{#2}
}
\providecommand{\href}[2]{#2}


\end{document}